\documentclass[11pt]{amsart}
\usepackage{color}
\usepackage{amsmath}
\usepackage{amsxtra}
\usepackage{amscd}
\usepackage{amsthm}
\usepackage{amsfonts}
\usepackage{amssymb}
\usepackage{eucal}
\usepackage{epsfig}
\usepackage{graphics}
\usepackage[all]{xy}

\def\bdeux{2}
\def\bun{1}

\theoremstyle{plain}
\newtheorem{theorem}{Theorem}[section]

\newtheorem{thm}[theorem]{Theorem}

\newtheorem{prop}[theorem]{Proposition}
\newtheorem{lem}[theorem]{Lemma}
\newtheorem{conj}[theorem]{Conjecture}
\newtheorem{defi}[theorem]{Definition}
\newtheorem{example}[theorem]{Example}

\def\tG{\widetilde{\Gamma}}
   
\newcommand{\Glie}{\mathfrak{g}}
\newcommand{\Yim}{\mathcal{Y}}
\newcommand{\A}{\mathcal{A}}

\newcommand{\ZZ}{\mathbb{Z}}
\newcommand{\CC}{\mathbb{C}}
\newcommand{\NN}{\mathbb{N}}

\newcommand{\QQ}{\mathbb{Q}}
\newcommand{\G}{\Gamma}
\usepackage[all]{xy}

\newcommand{\C}{\mathbb{C}}
\newcommand{\Q}{\mathbb{Q}}
\newcommand{\Z}{\mathbb{Z}}

\newcommand{\g}{\mathfrak{g}}
\newcommand{\bo}{\mathfrak{b}}

\newcommand{\gb}{\dot{\mathfrak{g}}}

\newcommand{\Psib}{\mbox{\boldmath$\Psi$}}
\newcommand{\Psibs}{\scalebox{.7}{\boldmath$\Psi$}}
\newcommand{\qbin}[2]{{\left[
\begin{matrix}{\,\displaystyle #1\,}\\
{\,\displaystyle #2\,}\end{matrix}
\right]
}}

\newcommand{\nc}{\newcommand}
\nc{\on}{\operatorname}
\nc{\la}{\lambda}
\nc{\wh}{\widehat}
\nc{\supp}{\mathrm{supp}}

\def\sl{\mathfrak{sl}}

\newtheorem{rem}[theorem]{Remark}

\begin{document}

\title
[Cluster algebras and 
Borel subalgebras of quantum affine 
algebras]
{Cluster algebras and 
category $\mathcal{O}$ for 
representations of
Borel subalgebras of quantum affine 
algebras}

\author{David Hernandez and Bernard Leclerc}

\begin{abstract} 
Let $\mathcal{O}$ be the category of representations of the Borel subalgebra of 
a quantum affine algebra introduced by Jimbo and the first author.
We show that the Grothendieck ring of a certain monoidal subcategory 
of $\mathcal{O}$ has the structure of a cluster algebra of infinite rank,
with an initial seed consisting of prefundamental representations.
In particular, the celebrated Baxter relations for the 
6-vertex model get interpreted as Fomin-Zelevinsky
mutation relations.

\end{abstract}

\maketitle

\setcounter{tocdepth}{1}
\tableofcontents

\section{Introduction}

Let $U_q(\Glie)$ be an untwisted quantum affine algebra (we assume throughout this paper that $q\in\C^*$ is not a root of unity).
M. Jimbo and the first author introduced \cite{HJ} a category $\mathcal{O}$ of representations of a Borel
subalgebra $U_q(\bo)$ of $U_q(\Glie)$. Finite-dimensional
representations of $U_q(\Glie)$ are objects in this category as well as the infinite-dimensional 
prefundamental representations of $U_q(\bo)$ constructed in \cite{HJ}. 
They are obtained as asymptotic limits of Kirillov-Reshetikhin
 modules, which form a family of simple finite-dimensional representations of $U_q(\mathfrak{g})$. 
These prefundamental representations, denoted by $L^+_{i,a}$ and $L^-_{i,a}$, are simple $U_q(\bo)$-modules parametrized by a complex number $a\in\CC^*$ 
and $1\leq i\leq n$, where 
$n$ is the rank of the underlying finite-dimensional simple Lie algebra.
Such prefundamental representations were first constructed explicitly for 
${\mathfrak g} = \widehat{\sl}_2$ by Bazhanov-Lukyanov-Zamolodchikov \cite{BLZ}, for $\widehat{\sl}_3$ by Bazhanov-Hibberd-Khoroshkin \cite{BHK} and 
for $\widehat{\sl}_n$ with $i = 1$ by Kojima \cite{Ko}. 

\medskip

The category $\mathcal{O}$ and the prefundamental representations were used by E. Frenkel and the first author \cite{FH} to prove
a conjecture of Frenkel-Reshetikhin on the spectra of quantum integrable systems. 
Let us recall that the partition function $Z$ of a quantum 
integrable system is crucial to understand its physical properties. 
It may be written in terms of the eigenvalues $\lambda_j$ of the transfer matrix $T$.
Therefore, to compute $Z$ one needs to find the spectrum of $T$.
 Baxter tackled this question in his seminal 1971 paper
\cite{Baxter} for the 
6-vertex and 8-vertex models. He observed moreover
that the eigenvalues $\lambda_j$ of $T$ have a very remarkable form
\begin{equation}    \label{relB}
\lambda_j = A(z) \frac{Q_j(zq^2)}{Q_j(z)} + D(z) \frac{Q_j(zq^{-2})}{Q_j(z)},
\end{equation}
where $q,z$ are parameters of the model (quantum and spectral),
 the functions $A(z)$, $D(z)$ are universal (in the sense that they are the same for all eigenvalues),
and $Q_j$ is a polynomial.
The above relation is now called \emph{Baxter's relation} (or Baxter's $TQ$
relation).
In 1998, Frenkel-Reshetikhin conjectured \cite{Fre} that the spectra of more general quantum integrable systems 
constructed from a representation $V$ of a quantum affine algebra $U_q(\Glie)$
have a similar form. (In this framework, the 6-vertex model is the particular case when
${\mathfrak g} = \widehat{\sl}_2$ and $V$ is irreducible of dimension 2.) 
One of the main steps in the proof of this conjecture \cite{FH} is to interpret
the expected generalized Baxter relations as algebraic identities in the Grothendieck ring
of the category ${\mathcal O}$ for $U_q(\bo)$
(see \cite{h4} for a short overview). 
For example, if ${\mathfrak g} = \widehat{\sl}_2$ and $V$ is the $2$-dimensional simple representation of $U_q(\Glie)$
with $q$-character $\chi_q(V) = Y_{1,a}+ Y_{1,aq^2}^{-1}$, one gets
the following categorical version of Baxter's relation (\ref{relB}):
\begin{equation} \label{relB2}
[V\otimes L^+_{1,aq}] = [\omega_1][L^+_{1,aq^{-1}}] + [-\omega_1][L^+_{1,aq^{3}}].
\end{equation}
(Here, $[\omega_1]$ and $[-\omega_1]$ denote the classes of certain one-dimensional representations of $U_q(\bo)$, see
Definition~\ref{oned} below.) 

\medskip

In another direction, the notion of monoidal categorification
of cluster algebras was introduced by the authors in \cite{HL0}. The cluster
algebra $\mathcal{A}(Q)$ attached to a quiver $Q$ is a commutative $\Z$-algebra
with a distinguished set of generators called cluster variables and obtained
inductively via the Fomin-Zelevinsky procedure of mutation \cite{FZ1}. 
By definition, the rank of $\mathcal{A}(Q)$ is the number of vertices of $Q$ (finite or infinite).
A monoidal category $\mathcal{C}$ is said to be a \emph{monoidal categorification} of $\mathcal{A}(Q)$ if
there exists a ring isomorphism 
$\mathcal{A}(Q)\stackrel{\sim}{\rightarrow} K_0(\mathcal{C})$
which induces a bijection between cluster variables and classes
of simple modules which are \emph{prime} (without non trivial tensor factorization) and \emph{real}
(whose tensor square is simple). Various examples of monoidal categorifications
have been established in terms of quantum affine algebras \cite{HL0, HL},
perverse sheaves on quiver varieties \cite{N, KQ, Qin}, and Khovanov-Lauda-Rouquier algebras \cite{kkko1, kkko2}.

\medskip
In this paper, we propose new monoidal categorifications of cluster algebras in terms of the category $\mathcal{O}$
of a Borel subalgebra $U_q(\bo)$ of an untwisted quantum affine algebra $U_q(\Glie)$. 
More precisely, in \cite{HL} we have attached to $\Glie$ a semi-infinite quiver $G^-$ and 
we proved that the cluster algebra $\mathcal{A}(G^-)$ is isomorphic to the Grothendieck ring
of a monoidal category $\mathcal{C}_\ZZ^-$ of finite-dimensional representations of $U_q(\Glie)$. 
Moreover, the classes of the Kirillov-Reshetikhin
modules in $\mathcal{C}_\ZZ^-$ are the images under this isomorphism of a subset of the cluster variables.
Let $\Gamma$ be the doubly-infinite quiver corresponding to $G^-$, as defined in
\cite[\S2.1.2]{HL}). 
The main result of this paper (Theorem \ref{isom}) is that the completed cluster algebra $\mathcal{A}(\Gamma)$ attached 
to this doubly-infinite quiver is isomorphic to the Grothendieck ring of a certain monoidal subcategory $\mathcal{O}^+_{2\ZZ}$  
of $\mathcal{O}$. This subcategory $\mathcal{O}^+_{2\ZZ}$ is generated by finite-dimensional representations
and positive prefundamental representations whose spectral parameters satisfy an integrality condition (see below
Definitions \ref{defineg}, \ref{defO+2Z} and Proposition \ref{inf}).
Moreover, the classes of the positive prefundamental representations form the cluster variables of an initial seed
of $\mathcal{A}(\Gamma)$.
In particular, when ${\mathfrak g} = \widehat{\sl}_2$ the counterparts (\ref{relB2}) of 
Baxter's relations (\ref{relB})
get interpreted as instances of Fomin-Zelevinsky mutation 
relations\footnote{The generalized Baxter relations of \cite{FH} can also be regarded as relations 
in the cluster algebra, but not as mutation relations since they involve more than $2$ terms in general.}. 
For general types, the first step mutation relations are interpreted as other remarkable relations 
in the Grothendieck ring $K_0(\mathcal{O}^+)$ (Formula (\ref{firstmutation})).

\medskip
Along the way we get interesting additional results, for instance (i) the construction of new asymptotic representations
beyond the case of prefundamental representations (Theorem~\ref{normlimit}), 
and (ii) the tensor factorization of arbitrary simple modules of 
$\mathcal{O}^+$ into prime modules when $\Glie = \widehat{\sl}_2$ (Theorem~\ref{fact}). 
%
We conjecture that the category $\mathcal{O}^+_{2\ZZ}$ is a monoidal 
categorification of the cluster algebra $\mathcal{A}(\G)$ (Conjecture \ref{conjgene}).
We prove this conjecture for $\Glie = \widehat{\sl}_2$ (Theorem \ref{psl2}).
An essential tool in several of our proofs is a duality $D$ between the Grothendieck rings of 
certain subcategories $\mathcal{O}^+$ and $\mathcal{O}^-$ of $\mathcal{O}$ (Theorem~\ref{stable}),
which maps classes of simple objects to classes of simple objects (Theorem~\ref{simpledual}).

\medskip
The paper is organized as follows. In Section \ref{bck} we give some background on quantum affine (or loop) algebras and their Borel subalgebras.
In Section \ref{catO} we review the main properties of the category $\mathcal{O}$ introduced in \cite{HJ} and we introduce the subcategories
$\mathcal{O}^+$ and $\mathcal{O}^-$ of interest for this paper (Definition \ref{newcato}). 
In Section \ref{clustsection} we state the main result on the isomorphism between $\mathcal{A}(\G)$ and 
the Grothendieck ring of $\mathcal{O}^+_{2\Z}$.  
In section \ref{catoprop} we establish relevant properties of $\mathcal{O}^+$, 
in particular we introduce and study the duality $D$ between $\mathcal{O}^+$ and $\mathcal{O}^-$. 
The proof of Theorem \ref{isom}
is given in Section \ref{mainproof}. In the concluding Section~\ref{secconj} we present the conjecture on monoidal categorifications
and we give various evidences supporting it, in particular the existence of asymptotic representations (Section \ref{limfin}). 
To conclude we present additional conjectural 
relations in $K_0(\mathcal{O}^+)$ extending the generalized Baxter relations of \cite{FH} (Conjecture \ref{genconj}).
\medskip
\medskip

The main results of this paper were presented in several conferences 
(Oberwolfach Workshop ``Enveloping algebras and geometric representation theory'' in May 2015, 
Conference ``Categorical Representation Theory and Combinatorics'' in Seoul in December 2015, 
Conference ``A bridge between representation theory and Physics'' in Canterbury in January 2016). 
An announcement was also published in the Oberwolfach Report \cite{h4}.

\medskip

{\bf Acknowledgement : } D. Hernandez is supported in part by the European Research Council under the European Union's 
Framework Programme H2020 with ERC Grant Agreement number 647353 QAffine.

\section{Quantum loop algebra and Borel algebras}\label{bck}

\subsection{Quantum loop algebra}\label{debut}

Let $C=(C_{i,j})_{0\leq i,j\leq n}$ be an indecomposable Cartan matrix
of non-twisted affine type.  We denote by $\g$ the Kac-Moody Lie
algebra associated with $C$.  Set $I=\{1,\ldots, n\}$, and denote by
$\gb$ the finite-dimensional simple Lie algebra associated with the
Cartan matrix $(C_{i,j})_{i,j\in I}$.  Let $\{\alpha_i\}_{i\in I}$,
$\{\alpha_i^\vee\}_{i\in I}$, $\{\omega_i\}_{i\in I}$,
$\{\omega_i^\vee\}_{i\in I}$, $\dot{\mathfrak{h}}$ be the simple
roots, the simple coroots, the fundamental weights, the fundamental
coweights and the Cartan subalgebra of $\gb$, respectively.  We set
$Q=\oplus_{i\in I}\Z\alpha_i$, $Q^+=\oplus_{i\in I}\Z_{\ge0}\alpha_i$,
$P=\oplus_{i\in I}\Z\omega_i$. We will also use $P_\Q = P\otimes \Q$ with
its partial ordering defined by $\omega\leq \omega'$ if and only if $\omega'-\omega\in Q^+$.
Let $D=\mathrm{diag}(d_0,\ldots, d_n)$
be the unique diagonal matrix such that $B=DC$ is symmetric and the
$d_i$'s are relatively prime positive integers.  We denote by
$(~,~):Q\times Q\to\Z$ the invariant symmetric bilinear form such that
$(\alpha_i,\alpha_i)=2d_i$.  We use the numbering of the Dynkin
diagram as in \cite{ka}.  Let $a_0,\ldots,a_n$ be the Kac labels
(\cite{ka}, pp.55-56). We have $a_0 = 1$ and we set $\alpha_0 =
-(a_1\alpha_1 + a_2\alpha_2 + \cdots + a_n\alpha_n)$.

Throughout this paper, we fix a non-zero complex number $q$ which is
not a root of unity.  We set $q_i=q^{d_i}$. We fix once and for all
$h\in\CC$ such that $q=e^h$, and we define $q^r=e^{rh}$ for any
$r\in\QQ$. Since $q$ is not a root of unity, for $r,s\in \QQ$ we have that
$q^r = q^s$ if and only if $r=s$.

We
will use the standard symbols for $q$-integers
\begin{align*}
[m]_z=\frac{z^m-z^{-m}}{z-z^{-1}}, \quad
[m]_z!=\prod_{j=1}^m[j]_z,
 \quad 
\qbin{s}{r}_z
=\frac{[s]_z!}{[r]_z![s-r]_z!}. 
\end{align*}

The \emph{quantum loop algebra} $U_q(\g)$ is the $\C$-algebra defined by generators 
$e_i,\ f_i,\ k_i^{\pm1}$ ($0\le i\le n$) 
and the following relations for $0\le i,j\le n$.
\begin{align*}
&k_ik_j=k_jk_i,\quad k_0^{a_0}k_1^{a_1}\cdots k_n^{a_n}=1,\quad
&k_ie_jk_i^{-1}=q_i^{C_{i,j}}e_j,\quad k_if_jk_i^{-1}=q_i^{-C_{i,j}}f_j,
\\
&[e_i,f_j]
=\delta_{i,j}\frac{k_i-k_i^{-1}}{q_i-q_i^{-1}},
\\
&\sum_{r=0}^{1-C_{i.j}}(-1)^re_i^{(1-C_{i,j}-r)}e_j e_i^{(r)}=0\quad (i\neq j),
&\sum_{r=0}^{1-C_{i.j}}(-1)^rf_i^{(1-C_{i,j}-r)}f_j f_i^{(r)}=0\quad (i\neq j)\,.
\end{align*}
Here we have set $x_i^{(r)}=x_i^r/[r]_{q_i}!$ ($x_i=e_i,f_i$). 
The algebra $U_q(\g)$ has a Hopf algebra structure given by
\begin{align*}
&\Delta(e_i)=e_i\otimes 1+k_i\otimes e_i,\quad
\Delta(f_i)=f_i\otimes k_i^{-1}+1\otimes f_i,
\quad
\Delta(k_i)=k_i\otimes k_i\,,
\\
&
S(e_i) = -k_i^{-1} e_i ,\quad S(f_i) = -f_i k_i,
\quad
S(k_i)=k_i^{-1}\,,
\end{align*}
where $i=0,\cdots,n$. 

The algebra $U_q(\g)$ can also be presented in terms of the Drinfeld
generators \cite{Dri2,bec}
\begin{align*}
  x_{i,r}^{\pm}\ (i\in I, r\in\Z), \quad \phi_{i,\pm m}^\pm\ (i\in I,
  m\geq 0), \quad k_i^{\pm 1}\ (i\in I).
\end{align*}
\begin{example} 
{\rm
For $\gb = \sl_2$, $e_1 = x_{1,0}^+$, $e_0 =
  k_1^{-1} x_{1,1}^-$, $f_1 = x_{1,0}^-$, $f_0 = x_{1,-1}^+ k_1$.
}
\end{example}
We shall use the generating series $(i\in I)$: 
$$\phi_i^\pm(z) = \sum_{m\geq 0}\phi_{i,\pm m}^\pm z^{\pm m} =
k_i^{\pm 1}\text{exp}\left(\pm (q_i - q_i^{-1})\sum_{m > 0} h_{i,\pm
    m} z^{\pm m} \right).$$
We also set $\phi_{i,\pm m}^\pm = 0$ for $m < 0$, $i\in I$.

The algebra $U_q(\g)$ has a $\ZZ$-grading defined by $\text{deg}(e_i)
= \text{deg}(f_i) = \text{deg}(k_i^{\pm 1}) = 0$ for $i\in I$ and
$\text{deg}(e_0) = - \text{deg}(f_0) = 1$.  It satisfies
$\text{deg}(x_{i,m}^\pm) = \text{deg}(\phi_{i,m}^\pm) = m$ for $i\in
I$, $m\in\ZZ$. For $a\in\CC^*$, there is a corresponding automorphism
$\tau_a : U_q(\g)\rightarrow U_q(\g)$ such that $\tau_a(x) = a^m x$
for every $x\in U_q(\g)$ of degree $m\in\ZZ$.

By \cite[Proposition 1.6]{Cha2}, there is an involutive automorphism
$\hat{\omega} : U_q(\g)\rightarrow U_q(\g)$ defined by ($i\in I$,
$m,r\in\Z$, $r\neq 0$)
$$\hat{\omega}(x_{i,m}^\pm) = - x_{i,-m}^\mp,\quad
\hat{\omega}(\phi_{i,\pm m}^\pm) = \phi_{i,\mp m}^{\mp},\quad
\hat{\omega}(h_{i,r}) = - h_{i,-r},\quad (i\in I,\ m,r\in\Z,\ r\neq 0).$$

Let $U_q(\g)^\pm$ (resp. $U_q(\g)^0$) be the subalgebra of $U_q(\g)$
generated by the $x_{i,r}^\pm$ where $i\in I, r\in\ZZ$ (resp. by the
$\phi_{i,\pm r}^\pm$ where $i\in I$, $r\geq 0$). We have a triangular
decomposition \cite{bec}
\begin{equation}\label{fdecomp}U_q(\g)\simeq U_q(\g)^-\otimes
  U_q(\g)^0 \otimes U_q(\g)^+.\end{equation}

\subsection{Borel algebra}\label{boralg}\label{subsec:Utg}

\begin{defi} The Borel algebra $U_q(\bo)$ is 
the subalgebra of $U_q(\g)$ generated by $e_i$ 
and $k_i^{\pm1}$ with $0\le i\le n$. 
\end{defi}
This is a Hopf subalgebra of $U_q(\g)$. 
The algebra $U_q(\mathfrak{b})$ 
contains 
the Drinfeld generators
$x_{i,m}^+$, $x_{i,r}^-$, $k_i^{\pm 1}$, $\phi_{i,r}^+$ 
where $i\in I$, $m \geq 0$ and $r > 0$. 

Let $U_q(\bo)^\pm = U_q(\g)^\pm\cap U_q(\bo)$ and
$U_q(\bo)^0 = U_q(\g)^0\cap U_q(\bo)$. 
Then we have 
\begin{align*}
U_q(\bo)^+ = \langle x_{i,m}^+\rangle_{i\in I, m\geq 0},\quad
U_q(\bo)^0 = 
\langle\phi_{i,r}^+,k_i^{\pm1}\rangle_{i\in I, r>0}.
\end{align*} 
It follows from \cite{bec, da} that we have a triangular decomposition
\begin{equation}\label{triandecomp}
U_q(\bo)\simeq U_q(\bo)^-\otimes 
U_q(\bo)^0 \otimes U_q(\bo)^+.
\end{equation}

\section{Representations of Borel algebras}\label{catO}

In this section we review results on representations of the Borel
algebra $U_q(\bo)$, in particular on the category $\mathcal{O}$
defined in \cite{HJ} and on finite-dimensional representations of
$U_q(\g)$. We also introduce the subcategories $\mathcal{O}^+$ and 
$\mathcal{O}^-$ of particular interest for this paper.

\subsection{Highest $\ell$-weight modules}\label{lhwm}
For a $U_q(\mathfrak{b})$-module $V$ and $\omega\in P_\Q$, we have the weight space
\begin{align}
V_{\omega}=\{v\in V \mid  k_i\, v = q_i^{\omega(\alpha_i^\vee)} v\ (\forall i\in I)\}\,.
\label{wtsp}
\end{align}
 We say that $V$ is 
\emph{Cartan-diagonalizable} 
if $V=\underset{\omega\in P_\Q}{\bigoplus}V_{\omega}$.

For $V$ a Cartan-diagonalizable $U_q(\bo)$-module, 
we define a structure of $U_q(\bo)$-module on its
graded dual $V^* = \bigoplus_{\beta\in P_\Q} V_\beta^*$ by
\begin{align*}
(x\,u)(v)=u\bigl(S^{-1}(x)v\bigr),\qquad
(u\in  V^*, \ v\in V,\ x \in U_q(\mathfrak{b})).
\end{align*}

\begin{defi} A series $\Psib=(\Psi_{i, m})_{i\in I,\, m\geq 0}$ 
of complex numbers such that 
$\Psi_{i,0}\in q_i^\Q$ for all $i\in I$ 
is called an $\ell$-weight. 
\end{defi}

We denote by $P_\ell$ the set of $\ell$-weights. 
Identifying $(\Psi_{i, m})_{m\geq 0}$ with its generating series we shall
write
\begin{align*}
\Psib = (\Psi_i(z))_{i\in I},
\quad
\Psi_i(z) = \underset{m\geq 0}{\sum} \Psi_{i,m} z^m.
\end{align*}
We will often use the involution $\Psib \mapsto \overline{\Psib}$ on $P_\ell$, where
$\overline{\Psib}$ is obtained from $\Psib$ by replacing every pole and zero of $\Psib$ by its inverse.


Since each $\Psi_i(z)$ is an invertible formal power series,
$P_\ell$ has a natural group structure
by component-wise multiplication. 
We have a surjective morphism of groups
$\varpi : P_\ell\rightarrow P_\Q$ given by 
$\Psi_i(0) = q_i^{\varpi(\Psib)(\alpha_i^\vee)}$.
For a $U_q(\mathfrak{b})$-module $V$ and $\Psib\in P_\ell$, 
the linear subspace
\begin{align}
V_{\Psibs} =
\{v\in V\mid
\exists p\geq 0, \forall i\in I, 
\forall m\geq 0,  
(\phi_{i,m}^+ - \Psi_{i,m})^pv = 0\}
\label{l-wtsp} 
\end{align}
is called the $\ell$-weight space of $V$ of $\ell$-weight $\Psib$. 
Note that since $\phi_{i,0}^+ = k_i$, we have $V_{\Psib} \subset V_\omega$, where
$\omega = \varpi(\Psib)$.

\begin{defi} A $U_q(\mathfrak{b})$-module $V$ is said to be 
of highest $\ell$-weight 
$\Psib\in P_\ell$ if there is $v\in V$ such that 
$V =U_q(\mathfrak{b})v$ 
and the following hold:
\begin{align*}
e_i\, v=0\quad (i\in I)\,,
\qquad 
\phi_{i,m}^+v=\Psi_{i, m}v\quad (i\in I,\ m\ge 0)\,.
\end{align*}
\end{defi}

The $\ell$-weight $\Psib\in P_\ell$ is uniquely determined by $V$. 
It is called the highest $\ell$-weight of $V$. 
The vector $v$ is said to be a highest $\ell$-weight vector of $V$.
For any $\Psib\in P_\ell$, there exists a simple 
highest $\ell$-weight module $L(\Psib)$ of highest $\ell$-weight
$\Psib$. This module is unique up to isomorphism.

For $\Psib$ an $\ell$-weight, we set $\widetilde{\Psib} = (\varpi(\Psib))^{-1}\Psib$
and we introduce the simple $U_q(\bo)$-module
\[
 \widetilde{L}(\Psib) := L\left(\widetilde{\Psib}\right).
\]
This is the simple $U_q(\bo)$-module obtained from $L(\Psib)$ by shifting all $\ell$-weights by $\varpi(\Psib)^{-1}$
(see \cite[Remark 2.5]{HJ}).

The submodule of $L(\Psib)\otimes L(\Psib')$ generated by the tensor
product of the highest $\ell$-weight vectors is of highest
$\ell$-weight $\Psib\Psib'$. In particular, $L(\Psib\Psib')$ is a
subquotient of $L(\Psib)\otimes L(\Psib')$.

\begin{defi}\cite{HJ}
For $i\in I$ and $a\in\CC^\times$, let 
\begin{align}
L_{i,a}^\pm = L(\Psib_{i,a}^{\pm 1})
\quad \text{where}\quad 
(\Psib_{i,a}^{\pm 1})_j(z) = \begin{cases}
(1 - za)^{\pm 1} & (j=i)\,,\\
1 & (j\neq i)\,.\\
\end{cases} 
\label{fund-rep}
\end{align}
\end{defi}
We call $L_{i,a}^+$ (resp. $L_{i,a}^-$) a positive (resp. negative)
\emph{prefundamental representation}.

\begin{defi}\label{oned}
For $\omega\in P_\Q$, let $[\omega] = L(\Psib_\omega)$ where $(\Psib_\omega)_i(z) = q_i^{\omega(\alpha_i^\vee)}$ $(i\in I)$.
\end{defi}
Note that the representation $[\omega]$ is $1$-dimensional with a
trivial action of $e_0,\cdots, e_n$. 

For $a\in\CC^\times$, the subalgebra $U_q(\mathfrak{b})$ is stable by $\tau_a$.
Denote its restriction to $U_q(\mathfrak{b})$ by the same letter.
Then the pullbacks of the  $U_q(\mathfrak{b})$-modules $L_{i,b}^\pm$ by $\tau_a$ 
is $L_{i,ab}^\pm$. 

\subsection{Category $\mathcal{O}$}\label{precato}

For $\lambda\in P_\Q$, we set $D(\lambda )=
\{\omega\in P_\Q \mid \omega\leq\lambda\}$.

\begin{defi}\cite{HJ} A $U_q(\mathfrak{b})$-module $V$ 
is said to be in category $\mathcal{O}$ if:

i) $V$ is 
Cartan-diagonalizable;

ii) for all $\omega\in P_\Q$ we have 
$\dim (V_{\omega})<\infty$;

iii) there exist a finite number of elements 
$\lambda_1,\ldots,\lambda_s\in P_\Q$ 
such that the weights of $V$ are in 
$\underset{j=1,\ldots, s}{\bigcup}D(\lambda_j)$.
\end{defi}

The category $\mathcal{O}$ is a monoidal category. 

\begin{rem} 
{\rm The definition of $\mathcal{O}$ is slightly different from that in \cite{HJ}
as we allow only rational powers of $q$ for the eigenvalues of $k_i$.
}
\end{rem}

\newcommand{\mfr}{\mathfrak{r}}
Let 
$P_\ell^{\mathfrak{r}}$ be the subgroup of $P_\ell$
consisting of $\Psib$ such that 
$\Psi_i(z)$ is a rational function of $z$ for any $i\in I$.

\begin{thm}\label{class}\cite{HJ} Let $\Psib\in P_\ell$. 
A simple object in the category $\mathcal{O}$ is of highest $\ell$-weight and 
the simple module $L(\Psib)$ is in category 
$\mathcal{O}$ if and only if $\Psib\in P_\ell^{\mathfrak{r}}$. Moreover, 
for $V$ in category $\mathcal{O}$, $V_{\Psib}\neq 0$ implies $\Psib\in P_\ell^{\mathfrak{r}}$.
\end{thm}

Given a map $c : P_\ell^{\mathfrak{r}}\rightarrow \ZZ$, consider its support
\[
 \supp(c) = \{\Psib \in  P_\ell^{\mathfrak{r}} \mid c(\Psib) \not = 0\}.
\]
Let
$\mathcal{E}_\ell$ 
be the additive group of maps
$c : P_\ell^{\mathfrak{r}}\rightarrow \ZZ$  
such that $\varpi(\supp(c))$ is contained in a finite union
of sets of the form $D(\mu)$, and such that 
for every $\omega\in P_\Q$, 
the set $\supp(c) \cap \varpi^{-1}(\{\omega\})$
is finite.
Similarly, let $\mathcal{E}$ 
be the additive group of maps $c : P_\Q \rightarrow \ZZ$ 
whose support is contained in 
a finite union of sets of the form $D(\mu)$. 
The map $\varpi$ is naturally extended to a surjective homomorphism 
$\varpi : \mathcal{E}_\ell\rightarrow \mathcal{E}$.

As for the category $\mathcal{O}$ of a classical Kac-Moody Lie algebra,
the multiplicity of a simple module in a module of our category $\mathcal{O}$
is well-defined (see \cite[Section 9.6]{ka}) and we have its Grothendieck
ring $K_0(\mathcal{O})$. Its elements are the formal sums 
$$\chi = \sum_{\Psib\in P_\ell^{\mathfrak{r}}} \lambda_{\Psib} [L(\Psib)]$$ 
where the $\lambda_{\Psib\in\ZZ}$ are set so that $\sum_{\Psib\in P_\ell^{\mathfrak{r}}, \omega\in P_\Q} |\lambda_{\Psib}| \text{dim}((L(\Psib))_\omega) [\omega]$
is in $\mathcal{E}$.

 We naturally identify $\mathcal{E}$ with 
the Grothendieck ring of the category of representations of $\mathcal{O}$ with constant $\ell$-weights, 
the simple objects of which are the $[\omega]$, $\omega\in P_\Q$. Thus as in \cite[Section 9.7]{ka} we will regard elements of $\mathcal{E}$
as formal sums
\[
 c = \sum_{\omega\in \supp(c)} c(\omega)[\omega].
\]
The multiplication is given by $[\omega][\omega'] = [\omega+\omega']$ and $\mathcal{E}$
is regarded as a subring of $K_0(\mathcal{O})$. 
If $(c_i)_{i\in \NN}$ is a countable family of elements of $\mathcal{E}$ 
such that for any $\omega\in P_\QQ$, $c_i(\omega)\neq 0$ for finitely
many $i\in\NN$, then the sum $\sum_{i\in\NN}c_i$ is well defined as a map $P_\QQ\rightarrow \ZZ$. 
When this map is in $\mathcal{E}$ we say that $\sum_{i\in\NN}c_i$ is a countable sum of elements in $\mathcal{E}$.
Note that we have the analog notion of a countable sum in $K_0(\mathcal{O})$, compatible
with countable sums of characters in $\mathcal{E}$.

\subsection{Finite-dimensional representations}

Let $\mathcal{C}$ be the category of (type $1$) finite-dimen\-sional
representations of $U_q(\Glie)$.

For $i\in I$, let $P_i(z)\in\mathbb{C}[z]$ be a polynomial with constant 
term $1$. Set
\begin{align*}
\Psib = (\Psi_i(z))_{i\in I},
\quad
\Psi_i(z) = 
q_i^{\text{deg}(P_i)}\frac{P_i(zq_i^{-1})}{P_i(zq_i)}.
\end{align*}
Then $L(\Psib)$ is finite-dimensional.  Moreover the action of
$U_q(\mathfrak{b})$ can be uniquely extended to an action of $U_q(\Glie)$, and any simple object in the
category $\mathcal{C}$ is of this form. Hence $\mathcal{C}$ is a subcategory of $\mathcal{O}$ and the inclusion functor preserves simple objects.

For $i\in I$ and $a\in \C^*$, we denote by $V_{i,a}$ the simple finite-dimensional representation associated with the polynomials
\[
 P_i(z) = 1-za, \qquad P_j(z) = 1 \quad (j\not = i).
\]
The modules $V_{i,a}$ are called the \emph{fundamental representations}.

\subsection{The categories $\mathcal{O}^+$ and $\mathcal{O}^-$} 

We introduce two new subcategories $\mathcal{O}^+$ and $\mathcal{O}^-$ of the category $\mathcal{O}$.

\begin{defi}\label{defineg} An $\ell$-weight is said to be positive (resp. negative) if it is a monomial in the following $\ell$-weights :
\begin{itemize}
\item the $Y_{i,a} = q_i \Psib_{i,aq_i}^{-1}\Psib_{i,aq_i^{-1}}$ where $i\in I$, $a\in\CC^*$, 

\item the $\Psib_{i,a}$ (resp. $\Psib_{i,a}^{-1}$) where $i\in I$, $a\in\CC^*$, 

\item the $[\omega]$, where $\omega\in P_\QQ$. 
\end{itemize}
\end{defi}

\begin{defi}\label{newcato} $\mathcal{O}^+$ (resp. $\mathcal{O}^-$) is the category of representations in $\mathcal{O}$ whose simple
constituents have a positive (resp. negative) highest $\ell$-weight.
\end{defi}

\begin{rem}\label{rem3.10} {\rm
(i) By construction, $\mathcal{O}^+$, $\mathcal{O}^-$ are stable by extensions. We will prove they are also stable by tensor products (Theorem \ref{stable}).

(ii) There are other remarkable subcategories of $\mathcal{O}$, for example the category $\widehat{\mathcal{O}}$
of representations of $U_q(\Glie)$ which belong to $\mathcal{O}$ as representations of $U_q(\bo)$. This category
$\widehat{\mathcal{O}}$ was introduced in \cite{Her04} and further studied in \cite{MY}.

(iii) One motivation of Definition~\ref{newcato} is that $\mathcal{O}^\pm$ contains $\mathcal{C}$ as well as the prefundamental
representations $L_{i,a}^\pm$. We have the following inclusion diagram :
$$\begin{array}{ccc}\mathcal{O}&\supset&\mathcal{O}^+,\mathcal{O}^-
\\\cup&&\cup
\\ \widehat{\mathcal{O}}&\supset&\mathcal{C}.
\end{array}
$$
Note that $\widehat{\mathcal{O}}$ is not contained in $\mathcal{O}^+$ or $\mathcal{O}^-$, 
and conversely neither $\mathcal{O}^+$ nor $\mathcal{O}^-$ is contained in $\widehat{\mathcal{O}}$. 
For instance, for $\Glie = \widehat{\sl}_2$, the representation $L\left(\frac{1-zq}{q - z}\right)$ 
is in the category $\widehat{\mathcal{O}}$ by \cite[Theorem 3.6]{MY},
but not in $\mathcal{O}^+$ or $\mathcal{O}^-$ 
because its highest $\ell$-weight has no factorization as in Equation (\ref{facthw}) below. 
On the other hand the prefundamental representations $L_{1,a}^\pm$ are in
the category $\mathcal{O}^\pm$ but not in  $\widehat{\mathcal{O}}$
(see \cite[section 4.1]{HJ} or \cite[Theorem 3.6]{MY}).

(iv) All generalized Baxter's relations established in \cite{FH} hold in the Grothendieck rings $K_0(\mathcal{O}^+)$
or $K_0(\mathcal{O}^-)$ (see Theorem \ref{relations} reminded below).

(v) The factorization of real simple modules in $\mathcal{O}$ into
prime representations is not unique, so the full category $\mathcal{O}$
is not a good candidate for the notion of monoidal categorification 
discussed in the introduction. For example for $\Glie = \widehat{\sl}_2$, 
it follows from \cite[Remark 4.3, Theorem 4.6]{MY} that 
$$L\left(q^{-5}\frac{1-q^4z}{1-q^{-6}z}\right)\otimes L\left(q^{-9}\frac{1-q^8z}{1-q^{-10}z}\right) 
\simeq L\left(q^{-7}\frac{1-q^4z}{1-q^{-10}z}\right) \otimes L\left(q^{-7}\frac{1-q^8z}{1-q^{-6}z}\right).$$
Moreover the tensor product is simple real and each simple factor has the character 
of a Verma module. Consequently each factor is not isomorphic to a tensor product 
of prefundamental representations and so is prime.}
\end{rem}

\section{Cluster algebras}\label{clustsection}

We state the main results of this paper.
We refer the reader to \cite{FZsurvey} and \cite{GSV} for an introduction
to cluster algebras, and for any un\-defined terminology.

\subsection{An infinite rank cluster algebra}\label{defclus}

Let us recall the infinite quiver $G$ introduced in \cite[Section 2.1.3]{HL}. 
Put $\widetilde{V} = I \times \Z$. $\widetilde{\Gamma}$ is the quiver with vertex set $\widetilde{V}$.
The arrows of $\tG$ are given by
\[
((i,r) \to (j,s)) 
\quad \Longleftrightarrow \quad
(C_{i,j}\not = 0 
\quad \mbox{and} \quad
s=r+d_i C_{i,j}).
\]
By \cite{HL}, the quiver $\widetilde{\Gamma}$ has two isomorphic connected components.
We pick one of the two isomorphic connected components of $\widetilde{\Gamma}$ and call it
$\Gamma$. The vertex set of $\Gamma$ is denoted by $V$.
A second labeling of the vertices of $\Gamma$ is deduced from the first one by means 
of the function $\psi$ defined by
\begin{equation}
 \psi(i,r) = (i, r+d_i),\qquad ((i,r)\in V).
\end{equation}
Let $W\subset I \times \Z$ be the image of $V$ under $\psi$.
We shall denote by $G$ the same quiver as $\Gamma$ but with vertices labeled by $W$.

By analogy with \cite[Section 2.2.1]{HL}, consider an infinite set of indeterminates 
$$
\mathbf{z} = \{z_{i,r}\mid (i,r)\in V\}
$$
over $\Q$.
Let $\mathcal{A}(\G)$ be the cluster algebra defined by the initial seed 
$(\mathbf{z}, \G)$.
Thus,  $\mathcal{A}(\G)$ is the subring of the field of rational functions
$\Q(\mathbf{z})$ generated by all the cluster variables, that is the elements obtained from
some element of $\mathbf{z}$ via a finite sequence of seed mutations.
Each element of $\mathcal{A}(\G)$ is a linear combination
of finite monomials in some cluster variables.
By the Laurent phenomenon \cite{FZ1}, $\mathcal{A}(\G)$ is contained in $\Z[z_{i,r}^{\pm 1}]_{(i,r)\in V}$.

For our purposes in this paper, it is always possible to 
work with sufficiently large finite subseeds 
of the seed $(\mathbf{z},\G)$, and replace $\mathcal{A}(\G)$ by the genuine cluster subalgebras
attached to them.
On the other hand, statements become nicer if we allow ourselves to
formulate them in terms of the infinite rank cluster algebra $\mathcal{A}(\G)$.

Define an $\mathcal{E}$-algebra homomorphism $\chi : \Z[z_{i,r}^{\pm 1}]\otimes_\ZZ\mathcal{E} \to \mathcal{E}$ by setting
\[
 \chi(z_{i,r}^{\pm 1}) = \left[\left(\frac{\mp r}{2d_i}\right)\omega_i\right], \qquad ((i,r)\in V).
\]
For $A\in \mathcal{A}(\G)\otimes_\ZZ\mathcal{E}$, we write $\chi(A) = \sum_\omega A_\omega [\omega]$ and 
$|\chi|(A) = \sum_\omega |A_\omega| [\omega]$. 
We will consider the completed tensor product
$$\mathcal{A}(\G)\hat{\otimes}_\ZZ \mathcal{E},$$ 
that is, the algebra of countable sums $\sum_{i\in\NN} A_i$ of elements $A_i \in \mathcal{A}(\G)\otimes_\ZZ \mathcal{E}$ such that
$\sum_{i\in\NN} |\chi|(A_i)$ is in $\mathcal{E}$ as a countable sum (as defined in Section \ref{precato}).
Note that in particular we have the analog notion of a countable sum in $\mathcal{A}(\G)\hat{\otimes}_\ZZ\mathcal{E}$.

\subsection{Main Theorem}

\begin{defi}\label{defO+2Z}
Define the category $\mathcal{O}^+_{2\ZZ}$ as the subcategory of representations in $\mathcal{O}^+$ whose simple
constituents have a highest $\ell$-weight $\Psib$ such that the roots and the poles of $\Psib_i(z)$ 
are of the form $q^r$ with $(i,r)\in V$.
\end{defi}

We will write for short $\ell_{i,a} = [L_{i,a}^+]$.

\begin{thm}\label{isom} The category $\mathcal{O}^+_{2\ZZ}$ is monoidal 
 and the identification 
\begin{equation}
\label{identi} z_{i,r}\otimes \left[\frac{r}{2d_i}\omega_i\right]  \equiv\ell_{i,q^r},\qquad ((i,r)\in V)
\end{equation} 
defines an isomorphism of $\mathcal{E}$-algebras
$$\mathcal{A}(\G)\hat{\otimes}_\ZZ \mathcal{E} \simeq K_0(\mathcal{O}^+_{2\ZZ})$$
compatible with countable sums.
\end{thm}

\begin{rem}\label{ident} 
{\rm
(i) The identification (\ref{identi}) gives an isomorphism of $\mathcal{E}$-algebras 
$$\mathcal{E}[z_{i,r}^{\pm 1}]_{(i,r)\in V} \stackrel{\sim}{\longrightarrow}\mathcal{E}[\ell_{i,q^r}^{\pm 1}]_{(i,r)\in V}.$$ 
which can be extended to countable sums as above. So the main point of the proof of Theorem~\ref{isom} will be to show that the subalgebra $\mathcal{A}(G)\hat{\otimes}_\ZZ \mathcal{E}$ is 
mapped to $K_0(\mathcal{O}^+_{2\ZZ})$ by this isomorphism.

(ii) As in the case of finite-dimensional representations, the description 
of the simple objects of $\mathcal{O}^+$ essentially reduces to the description of the simple
objects of $\mathcal{O}_{2\Z}^+$ (the decomposition explained in  \cite[\S3.7]{HL0} can be 
extended to our more general situation by using the asymptotic approach of \S\ref{limfin} below).
Hence the Grothendieck ring $K_0(\mathcal{O}_{2\Z}^+)$ contains all the interesting information
on $K_0(\mathcal{O}^+)$.
}
\end{rem}

The proof of Theorem~\ref{isom} will be given in \S\ref{mainproof}, using material presented in \S\ref{catoprop}. 

\section{Properties of the category $\mathcal{O}^+$}\label{catoprop}

\subsection{$q$-characters}\label{qcarsec} The $q$-character morphism was first considered in \cite{Fre} 
and is a very useful tool for our proofs.

Recall from \S\ref{precato} the notations $\mathcal{E}$ and $\mathcal{E}_\ell$.
Because of the support condition, we can endow $\mathcal{E}$ with
a ring structure defined by  
\[
 (c\cdot d)(\omega) = \sum_{\omega'+\omega''=\omega} c(\omega')d(\omega''),
 \qquad (c,d\in \mathcal{E},\ \omega\in P_\Q).
\]
Similarly, $\mathcal{E}_\ell$ also has a ring structure given by
\[
 (c\cdot d)(\Psib) = \sum_{\Psib'\Psib''=\Psib} c(\Psib')d(\Psib''),
 \qquad (c,d\in \mathcal{E}_\ell,\ \Psib\in  P_\ell^{\mathfrak{r}}),
\]
and such that $\varpi$ becomes a ring homomorphism.

For $\Psib\in P_\ell^{\mathfrak{r}}$ and $\omega\in P_\Q$, we define the delta functions
$[\Psib] = \delta_{\Psibs,.}\in\mathcal{E}_\ell$ and 
$[\omega] = \delta_{\omega,.}\in\mathcal{E}$, where as usual $\delta$ 
denotes the Kr\"onecker symbol. Note that the above multiplications give
\[
 [\Psib']\cdot[\Psib''] = [\Psib'\Psib''],\quad
 [\omega']\cdot[\omega''] = [\omega'+\omega''].
\]

Let $V$ be a $U_q(\mathfrak{b})$-module in category $\mathcal{O}$. 
We define \cite{Fre, HJ} the $q$-character and the character of $V$; 
\[
\chi_q(V) := 
\sum_{\Psibs\in P_\ell^{\mathfrak{r}}}  
\mathrm{dim}(V_{\Psibs}) [\Psib]\in \mathcal{E}_\ell,
\qquad
\chi(V) := \varpi(\chi_q(V)) =  \sum_{\omega\in P_\Q} 
\text{dim}(V_\omega) [\omega]\in\mathcal{E}.
\]
If $V\in \mathcal{O}$ has a unique $\ell$-weight $\Psib$
whose weight $\varpi(\Psib)$ is maximal,
we also consider its normalized $q$-character $\widetilde{\chi}_q(V)$
and normalized character $\widetilde{\chi}(V)$ defined by
\begin{align*}
\widetilde{\chi}_q(V) := [\Psib^{-1}]\cdot\chi_q(V),
\qquad 
\widetilde{\chi}(V) := \varpi(\widetilde{\chi}_q(V)).
\end{align*}
Note that 
\[
\chi_q\left(\widetilde{L}(\Psib)\right) = [\widetilde{\Psib}]\cdot\widetilde{\chi}_q(L(\Psib))
\not = \widetilde{\chi}_q(L(\Psib)).  
\]

\begin{prop}\cite{HJ} The $q$-character morphism 
$$
\chi_q : K_0(\mathcal{O})\rightarrow
\mathcal{E}_\ell,\quad [V]\mapsto \chi_q(V),
$$
is an injective ring morphism.
\end{prop}

Following \cite{Fre}, consider the ring of Laurent polynomials $\Yim =
\ZZ[Y_{i,a}^{\pm 1}]_{i\in I,a\in\CC^*}$ in the indeterminates
$\{Y_{i,a}\}_{i\in I, a\in \C^*}$.  Let $\mathcal{M}$ be the multiplicative group of
Laurent monomials in~$\Yim$. For example, for $i\in I$ and $a\in\CC^*$ define
$A_{i,a}\in\mathcal{M}$ by
$$
A_{i,a}=
Y_{i,aq_i^{-1}}Y_{i,aq_i}
\left(\prod_{j : C_{j,i} = -1}Y_{j,a}
\prod_{j : C_{j,i} = -2}Y_{j,aq^{-1}}Y_{j,aq}
\prod_{j : C_{j,i} = -3}Y_{j,aq^{-2}}Y_{j,a}Y_{j,aq^2}\right)^{-1}\,.
$$
For a monomial 
$m = \prod_{i\in I, a\in\CC^*}Y_{i,a}^{u_{i,a}} \in \mathcal{M}$, 
we consider its `evaluation at $\phi^+(z)$'. 
By definition it is the element 
$m(\phi(z))\in P_\ell^{\mathfrak{r}}$  
given by
$$m\bigl(\phi(z))=
\prod_{i\in I, a\in\CC^*}
\left(Y_{i,a}(\phi(z))\right)^{u_{i,a}}\text{ where }
\Bigl(Y_{i,a}\bigl(\phi(z)\bigr)\Bigr)_j
=\begin{cases}
\displaystyle{q_i\frac{1-a q_i^{-1}z}{1-aq_iz}}& (j=i),\\
1 & (j\neq i).\\
\end{cases}$$
This defines an injective group morphism 
$\mathcal{M}\rightarrow P_\ell^{\mathfrak{r}}$. 
We identify a monomial $m\in\mathcal{M}$ with its image in 
$P_\ell^{\mathfrak{r}}$. This is compatible with the notation $Y_{i,a}$
used in definition \ref{defineg}. Note that $\varpi(Y_{i,a}) = \omega_i$ and $\varpi(A_{i,a}) = \alpha_i$.

It is proved in \cite{Fre} that a finite-dimensional
$U_q(\Glie)$-module $V$ satisfies 
$V = \bigoplus_{m\in\mathcal{M}} V_{m\left(\phi(z)\right)}$.
In particular, $\chi_q(V)$ can be viewed  
as an element of $\Yim$.

A monomial $M\in\mathcal{M}$ is said to be dominant if
$M\in\ZZ[Y_{i,a}]_{i\in I, a\in\CC^*}$.  
Given a finite-dimensional simple $U_q(\Glie)$-module $L(\Psib)$,
there exists a dominant monomial
$M\in\mathcal{M}$ such that
$\Psib = M\bigl(\phi(z)\bigr)$.
 We will also denote $L(\Psib) = L(M)$.

For example, for $i\in I$, $a\in\CC^*$ and $k\geq 0$, we have the
Kirillov-Reshetikhin module
\begin{align}
W_{k,a}^{(i)} = L\left(Y_{i,a}Y_{i,aq_i^2}\cdots Y_{i,aq_i^{2(k-1)}}\right).
\label{KRmod}
\end{align}

\begin{example}\label{exkr} 
{\rm
In the case $\Glie = \widehat{\sl}_2$, we have ($k\geq
  0$, $a\in\CC^*$) \cite{Fre}:
$$\chi_q(W_{k,aq^{1-2k}}^{(1)}) = Y_{aq^{-1}}Y_{aq^{-3}}\cdots Y_{aq^{-2k+1}}
(1 + A_{1,a}^{-1} + A_{1,a}^{-1}A_{1,aq^{-2}}^{-1} 
+ \cdots + A_{1,a}^{-1} \cdots A_{1,aq^{-2(k - 1)}}^{-1}).
$$
}
\end{example}

\begin{thm}\label{formuachar}
(i) \cite{HJ, FH} For any $a\in\CC^*$, $i\in I$ we have
$$\chi_q(L_{i,a}^+) = \left[\Psib_{i,a}\right] \chi(L_{i,a}^+) = \left[\Psib_{i,a}\right] \chi(L_{i,a}^-),$$
where $\chi(L_{i,a}^+) = \chi(L_{i,a}^-)$ does not depend on $a$.

(ii) \cite{HJ} For any $a\in\CC^*$, $i\in I$ we have
$$\chi_q(L_{i,a}^-)\in \left[\Psib_{i,a}^{-1}\right] (1 + A_{i,a}^{-1}\ZZ[[A_{j,b}^{-1}]]_{j\in I, b\in\CC^*}).$$
\end{thm}

\begin{example} \label{ex-calcul}
{\rm
In the case $\Glie = \widehat{\sl}_2$, we have:
\[
\chi_q(L_{1,a}^+) = [(1 - za)]\sum_{r\geq 0} [-2r\omega_1]\text{ , }\chi_q(L_{1,a}^-) = \left[\frac{1}{(1-za)}\right]\sum_{r\geq 0}
  A_{1,a}^{-1}A_{1,aq^{-2}}^{-1}\cdots A_{1,aq^{-2(r-1)}}^{-1}. 
\]
Thus, although positive and negative prefundamental representations
have the same character
\[
\chi(L_{1,a}^+) =  \chi(L_{1,a}^-) = \sum_{r\geq 0} [-2r\omega_1],
\]
their $q$-characters are very different:
the normalized $q$-character $\widetilde{\chi}_q(L_{1,a}^+)$ is independent of the spectral parameter $a$, 
whereas $\widetilde{\chi}_q(L_{1,a}^-)$ does depend on $a$. 
}
\end{example}

\subsection{Baxter relations}

We can now state the \emph{generalized Baxter relations}:

\begin{thm}\label{relations}\cite[Theorem 4.8, Remark 4.10\footnote{The result in \cite{FH} is stated in terms of the $L_{i,a}^+$ and in terms of the $R_{i,a}^+$ such that $(R_{i,a}^+)^*\simeq L_{i,a}^-$. In the last case, the variable $Y_{i,a}$ has to be replaced by $[\omega_i][R_{i,aq_i^{-1}}]^+/[R_{i,aq_i}^+]$ to get $[V]$. That is why in the statement of Theorem \ref{relations} in terms of the $L_{i,a}^-$ the representations $[-\omega_i]\simeq [\omega_i]^*$ and $V^*$ appear.}]{FH} Let $V$ be a finite-dimensional representation of
  $U_q(\g)$. Replace in $\chi_q(V)$ each variable $Y_{i,a}$ by
  $[\omega_i]\frac{\left[L_{i,aq_i^{-1}}^+\right]}{\left[L_{i,aq_i}^+\right]}$ (resp. $[-\omega_i]\frac{\left[L_{i,aq_i^{-1}}^-\right]}{\left[L_{i,aq_i}^-\right]}$) and
  $\chi_q(V)$ by $[V]$ (resp. $[V^*]$). Then, multiplying by a common denominator, we get a
  relation in the Grothendieck ring $K_0(\mathcal{O})$.
\end{thm}

\begin{example}\label{baxsl2}
{\rm Taking $\Glie = \widehat{\sl}_2$ and $V= L(Y_{1,a})$, we obtain 
the above-mentioned relation (\ref{relB2}), namely,
$$[L(Y_{1,a})][L_{1,aq}^+] = [L_{1,aq^{-1}}^+][\omega_1] + [L_{1,aq^3}^+][-\omega_1],$$
and
$$[L(Y_{1,aq^2})][L_{1,aq}^-] = [L_{1,aq^{-1}}^-][-\omega_1] + [L_{1,aq^3}^-][\omega_1].$$
}
\end{example}

\begin{rem}
{\rm
By our main Theorem \ref{isom} (and its dual version, see Remark \ref{dualmain}), 
these generalized Baxter relations 
are interpreted as relations in the cluster algebra $\mathcal{A}(G)$.
Moreover when ${\mathfrak g} = \widehat{\sl}_2$ the original Baxter relation (\ref{relB2})
gets interpreted as a Fomin-Zelevinsky mutation relation.
}
\end{rem}

The right-hand side of a generalized Baxter relation is an $\mathcal{E}$-linear combination 
of classes of tensor products of prefundamental representations. As shown in \cite{FH},
this can be seen as a tensor product decomposition in $K(\mathcal{O})$. Indeed we have:

\begin{thm}\cite{FH}\label{stensor} Any tensor product of positive
  (resp. negative) prefundamental representations $L_{i,a}^+$
  (resp. $L_{i,a}^-$)
  is simple. 
\end{thm}

\subsection{$\ell$-characters of representations and duality}

Let $K_0^\pm$ be the $\mathcal{E}$-subalgebra of $K_0(\mathcal{O}^\pm)$ 
generated by the $[V_{i,a}]$, $[L_{i,a}^\pm]$ ($i\in I, a\in\CC^*$).

It follows from Theorem \ref{relations} that the fraction field of $\mathcal{E}[\ell_{i,a}]_{i\in I,a\in\CC^*}$
contains $K_0^\pm$. More precisely, each element of $K_0^\pm$ is a Laurent polynomial 
in $\mathcal{E}[\ell_{i,a}^{\pm 1}]_{i\in I,a\in\CC^*}$.
Note that the $\ell_{i,a}$ are algebraically independent (this is for example a consequence of Theorem \ref{stensor} which implies that the monomials in the $\ell_{i,a}$ are linearly independent in $K_0(\mathcal{O}^+)$). In particular the expansion in 
$\mathcal{E}[\ell_{i,a}^{\pm 1}]_{i\in I,a\in\CC^*}$ is unique and we can define the injective ring morphism
$$\chi_\ell : K_0^+\rightarrow \mathcal{E}[\ell_{i,a}^{\pm 1}]_{i\in I,a\in\CC^*}$$
which is called the $\ell$-character morphism. 

In the same way, by using the relations in Theorem \ref{relations} in terms of the prefundamental representations
$[L_{i,a}^-]$ and by setting $\ell_{i,a}' = [L_{i,a^{-1}}^-]$, we get an injective ring morphism
$$\chi_\ell' : K_0^-\rightarrow \mathcal{E}[(\ell_{i,a}')^{\pm 1}]_{i\in I,a\in\CC^*}$$

\begin{example} 
{\rm
We can reformulate Theorem \ref{relations} as follows. For a finite-dimensional representation $V$
of $U_q(\Glie)$, the $\ell$-character $\chi_\ell(V)$ (resp. $\chi_\ell'(V^*)$) is obtained from the 
$q$-character $\chi_q(V)$ by replacing each variable $Y_{i,a}$ by $[\omega_i]\ell_{i,aq_i^{-1}}\ell_{i,aq_i}^{-1}$ 
(resp. $[-\omega_i]\ell_{i,a^{-1}q_i}'(\ell_{i,a^{-1}q_i^{-1}}')^{-1}$). In fact, this can
  also be seen as a change of variables analogous to that of
  \cite[Section 5.2.2]{HL}. For example for $\Glie = \widehat{\sl}_2$, using Example \ref{baxsl2} we have
	\begin{equation}\label{compsl2}\chi_\ell([L(Y_{1,a})]) = \frac{[\omega_1]\ell_{1,aq^{-1}} + [-\omega_1]\ell_{1,aq^3}}{\ell_{1,aq}}
	\text{ and }\chi_\ell'([L(Y_{1,a})]) = \frac{[-\omega_1]\ell_{1,a^{-1}q^3}' + [\omega_1]\ell_{1,a^{-1}q^{-1}}'}{\ell_{i,a^{-1}q}'}.\end{equation}
}
  \end{example}

\begin{example}\label{exsl3} 
{\rm
Set $\Glie = \widehat{\sl}_3$. Let $\Psib = (\frac{1-zq^{-2}}{1-z},(1-qz)) = [-\omega_1]Y_{1,q^{-1}}\Psib_{2,q}$. 
Let us compute $\chi_\ell(L(\Psib))$.
For $k\geq 0$, let $W_k = \widetilde{L}(M_k)$ where $M_k = (Y_{2,q^{2k}}\cdots Y_{2,q^4}Y_{2,q^2})Y_{1,q^{-1}}$. We can prove as in \cite[Section 7.2]{HJ} that 
$(W_k)_{k\geq 0}$ gives rise to a limiting $U_q(\mathfrak{b})$-module $W_\infty$ whose $q$-character is 
$$\chi_q(W_\infty) = \Psib (1 + A_{1,1}^{-1}) \chi(L_{2,q}^+) = \chi_q(L_{2,q}^+)[-\omega_1]Y_{1,q^{-1}}+[-\omega_1 + \omega_2]Y_{1,q}^{-1}\chi_q(L_{2,q^{-1}}^+).$$
But it follows from Theorem \ref{normlimit} (ii) below that $L(\Psib)$ and $W_\infty$ have the same character 
$\text{lim}_{k\rightarrow +\infty}\chi(\widetilde{L}(M_k)) = \text{lim}_{k\rightarrow +\infty}\chi(\widetilde{L}(\overline{M_k}^{-1}))$. 
As $L(\Psib)$ is a subquotient of $W_\infty$, they are isomorphic.
Consequently
$$\chi_\ell(L(\Psib)) = \frac{\ell_{2,q}\ell_{1,q^{-2}} + [-\alpha_1] \ell_{1,q^2}\ell_{2,q^{-1}}}{\ell_{1,1}}.$$
}
\end{example}

\begin{prop}\label{fl} A representation $V$ in $\mathcal{O}^\pm$ satisfying $[V]\in K_0^\pm$ has finite length.
\end{prop}

\begin{rem}{\rm An object in the category $\mathcal{O}$ has not necessarily finite length.
The subcategory of objects of finite length is not stable by tensor product \cite[Lemma C.1]{BJMST}.}
\end{rem}

\begin{proof} Suppose $[V]\in K_0^+$. Then there is a monomial $M$ in the $\ell_{i,a}$
such that $M \chi_\ell(V)$ is a polynomial in the $\ell_{i,a}$. There is a tensor product $L$ of positive prefundamental
representations such that $\chi_\ell(L) = M$. Then $L\otimes V$ has finite length, hence the result.
\end{proof}

\begin{prop}\label{duality} The identification of the variables $\ell_{i,a}$ and $\ell_{i,a}'$ induces a unique isomorphism of $\mathcal{E}$-algebras 
$$D : K_0^+\rightarrow K_0^-.$$ 
\end{prop}

\begin{proof} The identification gives a well-defined injective ring morphism 
$$D' : K_0^+\rightarrow \mathcal{E}[(\ell_{i,a}')^{\pm 1}]_{i\in I, a\in\CC^*}.$$
It suffices to prove that its image is $\chi_l'(K_0^-)$. For $V = L_{i,b}^+$ a prefundamental,
its image by $D'$ is $\ell_{i,b}' = \chi_\ell'([L_{i,b^{-1}}^-])$. For $V = L(Y_{i,b})$ a fundamental, its image by
$D'$ is obtained from its $q$-character $\chi_q(V)$ by replacing each variable $Y_{i,a}$ by 
$[\omega_i]\frac{[L_{i,aq_i^{-1}}^-]}{[L_{i,aq_i}^-]}$, that is by
$[\omega_i]\frac{\ell_{i,a^{-1}q_i}'}{\ell_{i,a^{-1}q_i^{-1}}'}$. In the construction of 
$\chi'$, this corresponds to $Y_{i,a^{-1}}^{-1}$ (see the formula
in Theorem \ref{relations}). By \cite[Lemma 4.10]{Herbis},
there is a finite-dimensional representation $V'$ whose $q$-character is obtained from $\chi_q(V)$ by replacing each $Y_{i,a}$ by $Y_{i,a^{-1}}^{-1}$. Hence 
$D'(V) = \chi_\ell'((V')^*) \in \chi_\ell'(K_0^-)$. We have proved $\text{Im}(D') \subset \chi_\ell'(K_0(\mathcal{O}^-_{f,\ZZ}))$. 
Similarly, for $W$ such that $V = W^*$, we have $D'(W') = \chi_\ell'(V)$ and we get the other inclusion.
\end{proof}

\begin{example}\label{exdual}
{\rm
(i) For any $(i,a)\in I \times \CC^*$, we have 
$$D([L_{i,a}^+]) = [L_{i,a^{-1}}^-].$$

(ii) For any dominant monomial $m$, we have
$$D([L(m)]) = [L(m_1)]$$ 
where $m_1$ is obtained from $m$ 
by replacing each $Y_{i,a}$ by $Y_{i,a^{-1}}$. Indeed $L(m_1)$ is isomorphic
to $((L(m))')^*$ whose highest weight monomial is the inverse of the lowest weight 
monomial $(m_1)^{-1}$ of $L(m)'$. For example for $\Glie = \widehat{\sl}_2$, we have
$D([L(Y_{1,a})] = [L(Y_{1,a^{-1}})]$ as this can be observed by comparing the two formulas in (\ref{compsl2}).}
\end{example}

\begin{rem} 
{\rm
 The duality $D$ is compatible with characters by (i) in Theorem \ref{formuachar}.
However it is not compatible with $q$-characters (for example negative and positive prefundamental
representations have very different $q$-characters as explained above). 
}
\end{rem}

\begin{prop}\label{inf} An element in the Grothendieck group $K_0(\mathcal{O}^\pm)$ is a (possibly countable) sum of elements in $K_0^\pm$.
\end{prop}

\begin{proof} Let us prove it for $K_0(\mathcal{O}^+)$ (the proof is analog for $K_0(\mathcal{O}^-)$).
By definition, the positive $\ell$-weights label the simple modules in $\mathcal{O}^+$.
Moreover, an $\ell$-weight is positive if and only if it is a product of highest $\ell$-weights of representations 
$L_{i,a}^+$, $V_i(a)$ and $[\omega]$. 
This implies that for each positive $\ell$-weight $\Psib$, we can choose (and fix) a monomial
in the $[L_{i,a}^+]$, $[V_i(a)]$, $[\omega]$ such that the corresponding representation has highest $\ell$-weight equal to
$\Psib$.
Hence the positive $\ell$-weights also label the 
linearly independent family of these monomials in $K_0^+$. Expanding these monomials we get finite sums of classes
of simple modules by Proposition \ref{fl}. We get an (infinite) transition matrix
from the classes of simple objects in $\mathcal{O}^+$ to such products in $K_0^+$, and this matrix is unitriangular
(for the standard partial ordering with respect to weight). Hence the result.
\end{proof}

This implies immediately the following.

\begin{thm}\label{stable} $\mathcal{O}^+$ and $\mathcal{O}^-$ are monoidal and the morphism $D$ extends uniquely to an isomorphism of $\mathcal{E}$-algebras 
$$D : K_0(\mathcal{O}^+)\rightarrow K_0(\mathcal{O}^-).$$
\end{thm}

\begin{rem}\label{dualmain}{\rm Consequently, our main result in Theorem \ref{isom} may also be written in terms
of the subcategory $\mathcal{O}^-_{2\ZZ}$ of $\mathcal{O}^-_\ZZ$ whose Grothendieck ring is $D(K_0(\mathcal{O}^+_{2\ZZ}))$.}
\end{rem}

\begin{proof} For $L$, $L'$ simple in $\mathcal{O}^+$, we may consider a decomposition of $[L]$, $[L']$ as a countable sum of elements  in $K_0^+$ as in Proposition \ref{inf}. Then $[L][L']$ is also such a countable sum and is in $K_0(\mathcal{O}^+)$.
Hence $\mathcal{O}^+$ is monoidal. This is analog for $\mathcal{O}^-$. 

The isomorphism of Proposition \ref{duality} is extended by linearity to $K_0(\mathcal{O}^+)$ by using Proposition \ref{inf}. This map $D : K_0(\mathcal{O}^+)\rightarrow K_0(\mathcal{O}^-)$ is an injective ring morphism. The ring morphism $D^{-1}:K_0(\mathcal{O}^-)\rightarrow K_0(\mathcal{O}^+)$ is constructed in the same way and so $D$ is a ring isomorphism.
\end{proof}

\begin{prop}\label{debutprop} 
A simple object in $\mathcal{O}^\pm$ is a subquotient of a tensor product of two simple representations $V\otimes L$ where $V$ is finite-dimensional and $L$ is a tensor product of positive (resp. negative) prefundamental representations.
\end{prop}

\begin{proof} 
Let $L(\Psib)$ be simple in $\mathcal{O}^\pm$. By definition, its highest $\ell$-weight is a product
of highest $\ell$-weights of representations $[\omega]$, $L_{i,a}^\pm$, $V_i(a)$ where $\omega\in P_\Q$, $i\in I$,
$a\in\CC^*$. So $\Psib$ can be factorized in
\begin{equation}\label{facthw}\Psib = [\omega] \times m \times \prod_{i\in I, a\in\CC^*}\Psib_{i,a}^{u_{i,a}},\end{equation}
where $\omega\in P_\Q$, $\pm u_{i,a}\geq 0$ and $m\in\mathcal{M}$ is a dominant monomial. 
The result follows by taking
$V = [\omega]\otimes L(m)$ and
$L = \bigotimes_{i\in I, a\in\CC^*}(L_{i,a}^{\pm})^{\otimes | u_{i,a}|}$ which is simple by Theorem \ref{stensor}.\end{proof}

\begin{prop}\label{cone} The normalized $q$-character of a simple object in $\mathcal{O}^-$ belongs to
the ring $\ZZ[[A_{i,a}^{-1}]]_{i\in I, a\in\CC^*}$.
\end{prop}

\begin{proof} The result is known for the category $\mathcal{C}$ by \cite{Fre2}. For negative prefundamental representations, the result is known by \cite[Theorem 6.1]{HJ}. 
Then the general result follows from Proposition \ref{debutprop}.
\end{proof}

Note that this property is not satisfied in $\mathcal{O}^+$, see Example~\ref{ex-calcul}.

\section{Proof of the main Theorem}\label{mainproof}

\subsection{Examples of mutations}\label{mutaex}

\subsubsection{} \label{example-sl2}
Let $\g = \widehat{\sl}_2$. We display a sequence of 3 mutations starting from the initial seed of $\A(G)$.
The mutated cluster variables are indicated by a framebox. 
$$\xymatrix{ {}\save[]+<0cm,2ex>*{\vdots}\restore&{}\save[]+<0cm,2ex>*{\vdots}\restore&{}\save[]+<0cm,2ex>*{\vdots}\restore&{}\save[]+<0cm,2ex>*{\vdots}\restore
\\\ar[u] z_4&\ar[u] z_4&\ar[u] z_4&\ar[u]\ar[d] z_4
\\\ar[u] z_2&\ar[u] z_2\ar[d]&\ar[u] \fbox{$z_2$}\ar@/^{2pc}/[dd]&[L(Y_{q^{-3}}Y_{q^{-1}}Y_q)]\ar@/^{4pc}/[ddd]
\\\ar[u]\fbox{$z_0$}&\ar[d][L(Y_{q^{-1}})]&[L(Y_{q^{-1}})]&[L(Y_{q^{-1}})]
\\\ar[u]z_{-2}&\ar@/^{2pc}/[uu]\fbox{$z_{-2}$}&[L(Y_{q^{-1}}Y_{q^{-3}})]\ar[u]\ar[d]&[L(Y_{q^{-1}}Y_{q^{-3}})]\ar[u]\ar@/^{2pc}/[uu]
\\\ar[u]z_{-4}&\ar[u]z_{-4}&z_{-4}\ar@/^{3pc}/[uuu]&z_{-4}\ar@/^{4pc}/[uuuu]
\\{}\save[]+<0cm,4ex>*{\vdots}\restore&{}\save[]+<0cm,4ex>*{\vdots}\restore&{}\save[]+<0cm,4ex>*{\vdots}\restore&{}\save[]+<0cm,4ex>*{\vdots}\restore}$$

\subsubsection{} \label{example-sl3} Let $\g = \widehat{\sl}_3$. Example \ref{exsl3} can be reformulated as a one-step mutation from the
initial seed, as follows:
$$
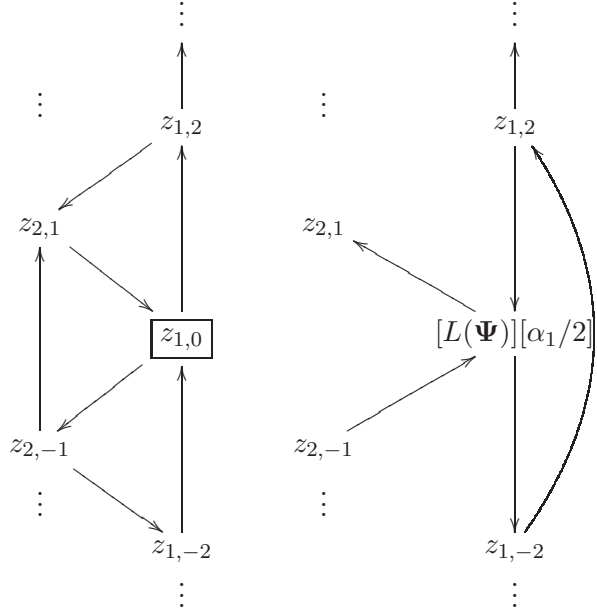
\xymatrix{ &{}\save[]+<0cm,2ex>*{\vdots}\restore&&{}\save[]+<0cm,2ex>*{\vdots}\restore
\\{}\save[]+<0cm,2ex>*{\vdots}\restore&\ar[u]\ar[dl] z_{1,2} & {}\save[]+<0cm,2ex>*{\vdots}\restore  &\ar[u]\ar[dd] z_{1,2}
\\z_{2,1}\ar[dr]&&z_{2,1}&
\\&\fbox{$z_{1,0}$}\ar[uu]\ar[dl]&&[L(\Psib)][\alpha_1/2]\ar[ul]\ar[dd]
\\z_{2,-1}\ar[uu]\ar[dr]&&z_{2,-1}\ar[ru]&
\\{}\save[]+<0cm,4ex>*{\vdots}\restore&z_{1,-2}\ar[uu]&{}\save[]+<0cm,4ex>*{\vdots}\restore&z_{1,-2}\ar@/_{2.5pc}/[uuuu]
\\&{}\save[]+<0cm,4ex>*{\vdots}\restore&&{}\save[]+<0cm,4ex>*{\vdots}\restore}$$
Recall that here $\Psib = [-\omega_1]Y_{1,q^{-1}}\Psib_{2,q}$, as in Example \ref{exsl3}.

\subsubsection{}
For an arbitrary $\g$, let us calculate the first mutation relation for each cluster variable $z_{i,r}$ of the initial seed,
generalizing \ref{example-sl3}. 
We denote by $z_{i,r}^*$ the new cluster variable obtained by mutating $z_{i,r}$. Then we have
$$z_{i,r}^*z_{i,r} = \prod_{j,C_{j,i}\neq 0} z_{j,r-d_jC_{j,i}} + \prod_{j,C_{j,i}\neq  0}z_{j,r + d_jC_{j,i}}.$$
We claim that $z_{i,r}^* = [\lambda][L(\Psib)]$ where 
$$\Psib = [-\omega_i]Y_{i,q^{r-d_i}}\prod_{j,C_{j,i}<0}\Psib_{j,q^{r-d_jC_{j,i}}},\text{ and }\lambda = \frac{\alpha_i}{2} - r \sum_{j,C_{j,i} < 0}\frac{\omega_j}{2 d_j}.$$
As in \ref{example-sl3}, this is derived from the explicit $q$-character formula
\begin{equation}\label{expqchar}\chi_q(L(\Psib)) = \left[\Psib\right](1 +  A_{i,q^r}^{-1})\prod_{j,C_{j,i}<0}\chi_j,\end{equation}
where $\chi_j = \chi(L_{j,a}^+)$ does not depend on $a$. (By considering $L(\Psib)\otimes L(Y_{i,q^{r-d_i}}\Psib^{-1})$
and $L(\Psib)\otimes L_{i,q^r}^+$ we prove that the multiplicities in $\chi_q(L(\Psib))$ are larger than in the right-hand side 
of (\ref{expqchar}).
The reverse inequality is established by considering $L(\overline{M_R}^{-1})\otimes L(\Psib \overline{M_R})$ where the 
monomials $M_R$ are defined
for $\overline{\Psib}^{-1}$ as in  Theorem \ref{normlimit}.)
The mutation relation thus becomes the following relation in the Grothendieck ring $K_0(\mathcal{O}^+)$:
\begin{equation}\label{firstmutation}\left[L(\Psib)\otimes L_{i,q^r}^+\right] = \left[\bigotimes_{j,C_{j,i}\neq 0}L_{j,q^{r-d_jC_{j,i}}}^+\right] 
+ [-\alpha_i]\left[\bigotimes_{j,C_{j,i}\neq 0}L_{j,q^{r+d_jC_{j,i}}}^+\right].\end{equation}
By Theorem \ref{stensor}, the two terms on the right-hand side are simple.
Hence this is the decomposition of the class of the tensor product into simple modules.

\subsection{Proof of Theorem \ref{isom}}

We identify the $\mathcal{E}$-algebras $\mathcal{E}[z_{i,r}^{\pm 1}]_{(i,r)\in V}$ and $\mathcal{E}[\ell_{i,q^r}^{\pm 1}]_{(i,r)\in V}$ as in Remark \ref{ident}. 

\medskip

\begin{prop}\label{unsens}
We have:  $K_0^+\cap K_0(\mathcal{O}^+_{ 2\ZZ}) \subseteq \mathcal{A}(G)\otimes_\ZZ \mathcal{E}$.
\end{prop}

Note that by Proposition \ref{inf}, this implies $K_0(\mathcal{O}^+_{ 2\ZZ})\subseteq \mathcal{A}(G)\hat{\otimes}_\ZZ \mathcal{E}$

\begin{proof}
Clearly, $\ell_{i,q^s} = z_{i,s}[(s/2d_i)\omega_i]$ belongs to $\mathcal{A}(G)_\ZZ \otimes_\ZZ \mathcal{E} $ for any $(i,s)\in V$.
By Proposition \ref{inf}, it remains to show that $\left[V_{i,q^{r}}\right]$ belongs to $\mathcal{A}(G)_\ZZ \otimes_\ZZ \mathcal{E} $ for any $(i,r)\in W$.

Remember from \S\ref{defclus} that we denote by $G$ the same quiver as $\G$ with vertices labeled by $W$ instead of $V$. 
In the next discussion, we divide the vertices of $G$ and $\G$ into \emph{columns}, as in \cite[Example 2.3]{HL}, and we denote
by $k$ the number of columns.
As in \cite[\S2.1.3]{HL}, consider the full subquiver $G^-$ of $G$ whose vertex set is $W^-=\{(i,r)\in W \mid r\le 0\}$. 
The definition of $G$ shows that there is only 
one vertex of $G\setminus G^-$ in each column which is connected to $G^-$ by some arrow. Let $H^-$ denote the ice quiver
obtained from $G^-$ by adding these $k$ vertices together with their connecting arrows, and by declaring the new vertices frozen.

Consider the cluster algebras $\A(H^-)$ (with $k$ frozen variables) and $\A(G^-)$ (with no frozen variable).
It follows from the definitions that $\A(H^-)$ can be regarded as a subalgebra of $\A(G)$, and $\A(G^-)$
is the coefficient-free counterpart of $\A(H^-)$, studied in \cite{HL}.
Let $f_j \ (1\le j\le k)$ be the frozen variable of $\A(H^-)$ sitting in column $j$. The cluster of the initial seed
of $\A(H^-)$ thus consists of the frozen variables $f_j \ (1\le j\le k)$ and the ordinary cluster variables $z_{i,s}\ ((i,s)\in V^-)$,
where $V^- = \{(i,s)\in V \mid s+d_i\le 0\}$.
Let us denote by $u_{i,s}\ ((i,s) \in V^-)$ the cluster variables of the initial seed of $\A(G^-)$.
We can use a similar change of variables as in \cite[\S2.2.2]{HL}:
\[
y_{i,r} = u_{i,r-d_i} \quad\mbox{if $r+d_i > 0$,}
\qquad
y_{i,r} = \frac{u_{i,r-d_i}}{u_{i,r+d_i}} \quad \mbox{otherwise}.
\]
Let 
\[
F\colon \Z[ u_{i,s}^{\pm 1} \mid(i,s) \in V^-] \longrightarrow \Z[f_j^{\pm1},z_{i,s}^{\pm1} \mid  1\le j\le k,\ (i,s)\in V^-]
\]
be the ring homomorphism defined by 
\[
F(u_{i,s}) = \frac{z_{i,s}}{f_j} \quad \mbox{if $(i,s)$ sits in column $j$.} 
\]
Thus
\[
F(y_{i,r}) = \frac{z_{i,r-d_i}}{f_j} \ \mbox{if $r+d_i > 0$ and $(i,r)$ sits in column $j$,}
\quad
F(y_{i,r}) = \frac{z_{i,r-d_i}}{z_{i,r+d_i}} \ \mbox{otherwise}.
\]
We introduce a $\Z^k$-grading on $\Z[f_j^{\pm1},z_{i,s}^{\pm1} \mid  1\le j\le k,\ (i,s)\in V^-]$ by declaring that
\[
 \deg(f_j) = e_j, \quad \deg(z_{i,s}) = e_j \ \mbox{if $(i,s)$ sits in column $j$,} 
\]
where $(e_j, 1\le j\le k)$ denotes the canonical basis of $\Z^k$.

Let $x$ be the cluster variable of $\A(G^-)$ obtained from the initial seed $(\{u_{i,s}\}, G^-)$ via a sequence of mutations 
$\sigma$, and let $y$ be the cluster variable of $\A(H^-)$ obtained from the initial seed $(\{z_{i,s}, f_j\}, H^-)$ via 
the \emph{same} sequence of mutations $\sigma$. We want to compare the Laurent polynomials $y$ and $F(x)$.
Since $\deg(F(u_{i,s}))=(0,\ldots,0)$ for every $(i,s)$, we see that $F(x)$ is multi-homogeneous of degree $(0,\ldots,0)$ for the above grading. 
On the other hand, it is easy to check that for every non-frozen vertex $(i,s)$ of the ice quiver $H^-$ the sum of the multi-degrees
of the initial cluster variables and frozen variables sitting at the targets of the arrows going out of $(i,s)$ is equal to
the sum of the multi-degrees
of the initial cluster variables and frozen variables sitting at the sources of the arrows going into $(i,s)$.
Therefore, $\A(H^-)$ is a multi-graded cluster algebra, in the sense of \cite{GL}.
It follows that $y$ is also multi-homogeneous, of degree $(a_1,\ldots,a_k)$.
Now, by construction, we have
\[
 F(x)|_{f_1=1,\ldots, f_k=1} = y|_{f_1=1,\ldots, f_k=1}.
\]
Therefore, 
\[
 y = F(x)\prod_{j=1}^k f_j^{a_j}.
\]
Taking a cluster expansion with respect to the initial cluster of $\A(H^-)$, we write
$y = N/D$ where $D$ is a monomial in the non-frozen cluster variables and $N$ is a polynomial
in the non-frozen and frozen variables. Moreover $N$ is not divisible by any of the $f_j$'s.
It follows that $\prod_{j=1}^k f_j^{a_j}$ is the smallest monomial such that 
$F(x)\prod_{j=1}^k f_j^{a_j}$ contains only nonnegative powers of the variables $f_j$'s.

Now we can conclude using \cite[Theorem 3.1]{HL}, which implies that for all $(i,r)\in W^-$
with $r \ll 0$, the $q$-character of $V_{i,q^r}$ (expressed in terms of the variables
$y_{i,s}\equiv Y_{i,q^s}$) is a cluster variable $x$ of $\A(G^-)$. 
By \cite[Corollary 6.14]{Fre2}, for $r\ll 0$
this cluster variable does not contain any variable $y_{i,s}$ with $s+2d_i>0$,
hence $F(x)$ does not contain any frozen variable $f_j$. Therefore $y = F(x)$, and
the $q$-character of $V_{i,q^r}$ (expressed in terms of the variables $z_{i,s}$) is a cluster
variable of $\A(H^-)$, that is, a cluster variable of $\A(\G)$. This proves the claim for every
fundamental module $V_{i,q^r}$ with $r \ll 0$. But by definition of the cluster algebra $\A(\G)$, the set 
of cluster variables is invariant under the change of variables $z_{i,s} \mapsto z_{i,s+2d_i}$.
Thus we are done.
\end{proof}

\begin{prop}\label{reciproque}
We have:  $\mathcal{A}(G)\hat{\otimes}_\ZZ \mathcal{E} \subseteq K_0(\mathcal{O}^+_{2\ZZ})$.
\end{prop}

Consider an element $\chi$ in $\mathcal{A}(G)$. 
By the Laurent phenomenon \cite{FZ1}, $\chi$ is a Laurent polynomial in the initial cluster variables :
$$\chi = P(\{z_{i,r}\}_{(i,r)\in V}).$$
Hence $\mathcal{A}(G)$ is a subalgebra of the fraction field of $K_0(\mathcal{O}_{2\ZZ}^+)$ and the duality $D$ of Proposition \ref{duality} can be algebraically extended to $\mathcal{A}(G)$.
In particular we have
$$D(\chi) = P(\{D(z_{i,r})\}_{(i,r)\in V})\in \text{Frac}(K_0(\mathcal{O}_{2\ZZ}^-))$$
in the fraction field of $K_0(\mathcal{O}_{2\ZZ}^-)$. The $q$-character morphism can also 
be algebraically extended to $\text{Frac}(K_0(\mathcal{O}_{2\ZZ}^-))$. Then 
$\chi_q(D(\chi))$ is obtained by replacing each $z_{i,r}$ by the corresponding $q$-character
\begin{equation}\chi_q(D(z_{i,r})) = \left[\left(\frac{-r}{2d_i}\right)\omega_i\right] \chi_q(L_{i,q^{-r}}^-) = \left[\left(\frac{-r}{2d_i}\right)\omega_i\right]\Psib_{i,q^{-r}}^{-1} (1 + \mathcal{A}_{i,r}),\end{equation}
where by Theorem \ref{normlimit}, $\mathcal{A}_{i,r}$ is a formal power series in the $A_{j,a}^{-1}$ without constant term.
 In particular, we have an analog
formula for the inverse which is 
$$(\chi_q(D(z_{i,r})))^{-1} = \left[\left(\frac{r}{2d_i}\right)\omega_i\right]\Psib_{i,q^{-r}} (1 + \mathcal{B}_{i,r})$$
where 
$$\mathcal{B}_{i,r} = \sum_{k\geq 1} (-\mathcal{A}_{i,r})^k$$
 is a formal power series in the $A_{j,a}^{-1}$ without constant term. In particular
$\chi_q(D(\chi))$ is in $\mathcal{E}_\ell$ and we get a sum of the form
\begin{equation}\label{laforme}
\chi_q(D(\chi)) = \sum_{1\leq \alpha\leq R} \lambda_\alpha [\omega_\alpha] m_\alpha (1 + \mathcal{A}_\alpha)\in\mathcal{E}_\ell
\end{equation}
where $\omega_\alpha$ is a weight, $m_\alpha$ a Laurent monomial in the $\Psib_{i,q^{-r}}$, $\mathcal{A}_\alpha$ is a formal power series in the $A_{j,a}^{-1}$ without constant 
term and $\lambda_\alpha\in \ZZ$. 

We recall the notion of negative $\ell$-weight is introduced in Definition \ref{defineg}. 

\begin{rem} {\rm
We say that a sequence $(\Psib^{(m)})_{m\geq 0}$ of $\ell$-weights congerges pointwise
as a rational fraction to an $\ell$-weight $\Psib$ if for every $i\in I$ and $z\in\CC$,
the ratio $\Psib^{(m)}_i(z)/\Psib_i(z)$ converges to $1$ when $N\rightarrow +\infty$ and $|q| > 1$.
For example, defining the monomials $M_{i,r,N}$ as in Equation (\ref{monkr}) below,
the sequence $(\widetilde{M_{i,r,N}})_{N\geq 0}$ converges pointwise as a rational fraction to $\Psib_{i,q^{-r}}^{-1}$.
}
\end{rem}

\begin{lem}\label{prneg} Let $\chi\in\mathcal{A}(G)$ non zero. Then at least one negative $\ell$-weight occurs in
$\chi_q(D(\chi))$.\end{lem}
 
\begin{proof}\label{maxdom} We will use the following partial ordering $\preceq$ on the set of $\ell$-weights 
$\Psib$ satisfying $\varpi(\Psib) = 1$ : for such $\ell$-weights $\Psib$, $\Psib'$, we set $\Psib \succeq \Psib'$ if 
$$\Psib' (\Psib)^{-1} = \prod_{i,r\geq -M}\widetilde{A_{i,q^r}}^{-v_{i,r}}$$
is a possibly infinite product (that is pointwise the limit of the partial products) with the $v_{i,r}\geq 0$. If $\Psib = \widetilde{m}$ and $\Psib' = \widetilde{m'}$ with $m, m'$ monomials in $\mathcal{M}$, $\Psib\preceq \Psib'$ is equivalent to $m\preceq m'$ for the partial ordering considered in \cite{Nab}.

As the sum (\ref{laforme}) is finite, there is $\alpha_0$ such that $m_{\alpha_0}$ is maximal for $\preceq$. We prove that $m_{\alpha_0}$ is a negative $\ell$-weight.

Let $N$ be such that all the cluster variables $z_{i,r}$ of the initial seed occurring in the
Laurent monomials of Equation (\ref{laforme}) satisfy $r > - 2d(N+2)$ where $d = \text{Max}_{i\in I}(d_i)$ is the lacing number of $\gb$. 
We consider as above the semi-infinite cluster 
algebra $\mathcal{A}(H^+_N)$ obtained from $\mathcal{A}(G)$ where the cluster variables sitting at $(i,r)\in V$, $r \leq - 2d(N + 2)$ 
have been removed. As explained in the proof of Proposition \ref{unsens}, $\mathcal{A}(H^+_N)$ can be regarded as a subalgebra
of $\mathcal{A}(G)$. We replace every cluster variable $z_{i,r}$ of the initial seed by the class of 
the Kirillov-Reshetikhin module $W_{i,r,N}$ of highest monomial
\begin{equation}\label{monkr}M_{i,r,N} = \prod_{k\geq 0 \text{ , }r + 2k d_i \leq 2dN} Y_{i,q^{-r - d_i - 2k d_i}}.\end{equation} 
(Here $M_{i,r,N}$ is set to be $1$ if $r > 2dN$).
We obtain 
$$\phi_N(\chi)\in \text{Frac}(K_0(\mathcal{C}))$$ 
the image of $\chi$ in the fraction field of $K_0(\mathcal{C})$.
By using the duality $D$, we reverse all spectral parameters (by (ii) in Example \ref{exdual} illustrating Proposition \ref{duality}, or by \cite[Section 3.4]{HL0}). 
We obtain the same\footnote{Note that we have a term $-d_i$ which does not occur in
\cite[Section 2.2.2]{HL} as we use the labeling by $V$ and not by $W$.} as in \cite[Section 2.2.2]{HL}.
Then by \cite[Theorem 5.1]{HL}, $D(\phi_N(\chi))$ belongs to the Grothendieck ring of $\mathcal{C}$. 
So by applying $D$ again, $\phi_N(\chi)$ is in $K_0(\mathcal{C})$.

Now we get as for Equation \ref{laforme} 
\begin{equation}\label{newlaforme}\chi_q(\phi_N(\chi)) = \sum_{1\leq \alpha\leq R}\lambda_\alpha m_\alpha^{(N)}(1 + \mathcal{A}_\alpha^{(N)})\end{equation}
where $m_\alpha^{(N)}$ is a monomial and $\mathcal{A}_\alpha^{(N)}$ is a formal power series in the $A_{i,a}^{-1}$.
Note that $\mathcal{A}_\alpha^{(N)}$ converges when $N\rightarrow +\infty$ to  $\mathcal{A}_\alpha$ as a formal power series in the $A_{i,a}^{-1}$.
Let $\omega_{i,r,N}$ be the highest weight of $W_{i,r,N}$. Now if the initial cluster variable $z_{i,r}$ is replaced by 
\begin{equation}\label{newident}\left[\left(\frac{r}{2d_i}\right)\omega_i -\omega_{i,r,N}\right][W_{i,r,N}]\end{equation} 
instead of $[W_{i,r,N}]$, we just have to replace in (\ref{newlaforme}) each $m_\alpha^{(N)}$ by $[\omega_\alpha]\widetilde{m_\alpha^{(N)}}$ (the $\mathcal{A}_\alpha^{(N)}$ are unchanged).
Then $\widetilde{m_\alpha^{(N)}}$ converges to $m_\alpha$ when
$N\rightarrow +\infty$ pointwise as a rational fraction. 

Let us show that there are infinitely many $N$ so that $\widetilde{m_{\alpha_0}^{(N)}}$ is maximal among the 
$\widetilde{m_\alpha^{(N)}}$ for $\preceq$.
Otherwise, since Equation \ref{laforme} has finitely many summands,
there is $\alpha$ such that $\widetilde{m_{\alpha_0}^{(N)}}\prec\widetilde{m_\alpha^{(N)}}$ for infinitely many $N$.  
In the limit, we get that $m_{\alpha_0} \prec m_\alpha$, contradiction. 

For $N$ such that $\widetilde{m_{\alpha_0}^{(N)}}$ is maximal for $\preceq$, $m_{\alpha_0}^{(N)}$ is necessarily
dominant as $\phi_N(\chi)\in K_0(\mathcal{C})$ (see \cite[Section 5.4]{Fre2}). 
Then the limit $m_{\alpha_0}$ of the $\widetilde{m_{\alpha_0}^{(N)}}$ is negative as it is easy to check  
that a limit of dominant monomials is a negative $\ell$-weight.
Finally, since $m_{\alpha_0}$ is maximal for $\preceq$, it  
necessarily occurs with a nonzero coefficient in the expansion of $\chi_q(D(\chi))$.
\end{proof}

\begin{lem} Let $\Psib$ be a negative $\ell$-weight such that the roots and the poles of $\Psib_i(z)$ are of the form $q^r$ with $(i,r)\in V$. Then there is a unique $F(\Psib)\in \chi_q(D(K_0(\mathcal{O}_{2\ZZ}^+)))$ 
such that $\Psib$ is the unique negative $\ell$-weight occurring in $F(\Psib)$ and its coefficient is $1$. 

Moreover $F(\Psib)$ is of the form
\begin{equation}\label{formo}F(\Psib) = [\Psib] + \sum_{\Psib',\varpi(\Psib')\prec \varpi(\Psib)} \lambda_{\Psib'} [\Psib'],\end{equation}
for the usual partial ordering on weights and with the $\lambda_{\Psib'}\in\ZZ$.
\end{lem}

\begin{proof} The uniqueness follows from Proposition \ref{unsens} and Lemma \ref{prneg}. For each negative $\ell$-weight
$\Psib$ as in the Lemma, there is a representation $M(\Psib)$ in $\mathcal{O}_{2\ZZ}^+$ such that $\chi_q(D([M(\Psib)]))$
is of the form
$$\chi_q(D([M(\Psib)])) = [\Psib] + \sum_{\Psib',\varpi(\Psib')\prec \varpi(\Psib)} \mu_{\Psib',\Psib} [\Psib'].$$
Indeed it suffices to consider a tensor product of fundamental and positive prefundamental representations. Now if the $F(\Psib)$ do exist, we have an infinite triangular transition matrix from the $(F(\Psib))$ to the $(\chi_q(D([M(\Psib)]))$ with $1$ on the diagonal and whose off-diagonal coefficients are the $\mu_{\Psib', \Psib}$ for $\Psib, \Psib'$ negative. So to prove the existence, it suffices to consider the inverse of this matrix (which is well-defined as for given $\Psib, \Psib'$, there is a finite number of $\Psib''$ satisfying $\varpi(\Psib')\preceq \varpi(\Psib'')$ and $\mu_{\Psib'',\Psib}\neq 0$).
\end{proof}

We can now finish the proof of Proposition \ref{reciproque} : 

\begin{proof} Let $\chi$ be in $\mathcal{A}(G)$. For $\Psib$ a negative $\ell$-weight, we denote by $\lambda_{\Psib}$
the coefficient of $\Psib$ in $\chi_q(D(\chi))$. Then by lemma \ref{prneg} we have
$$\chi_q(D(\chi)) = \sum_{\Psib\text{ negative}}\lambda_{\Psib} F(\Psib).$$
As the $F(\Psib)$ are of the form (\ref{formo}), this sum is well-defined in $\chi_q(D(K_0(\mathcal{O}_{2\ZZ}^+)))$
and we get $\chi\in K_0(\mathcal{O}_{2\ZZ}^+)$.
\end{proof}

\begin{example}{\rm Let us illustrate the proof of Lemma \ref{prneg} which is the crucial technical point for the proof of Proposition \ref{reciproque}.
Consider the sequence of mutations
of  \ref{example-sl2}. Let us write the cluster variables $\phi_N(\chi)$ :
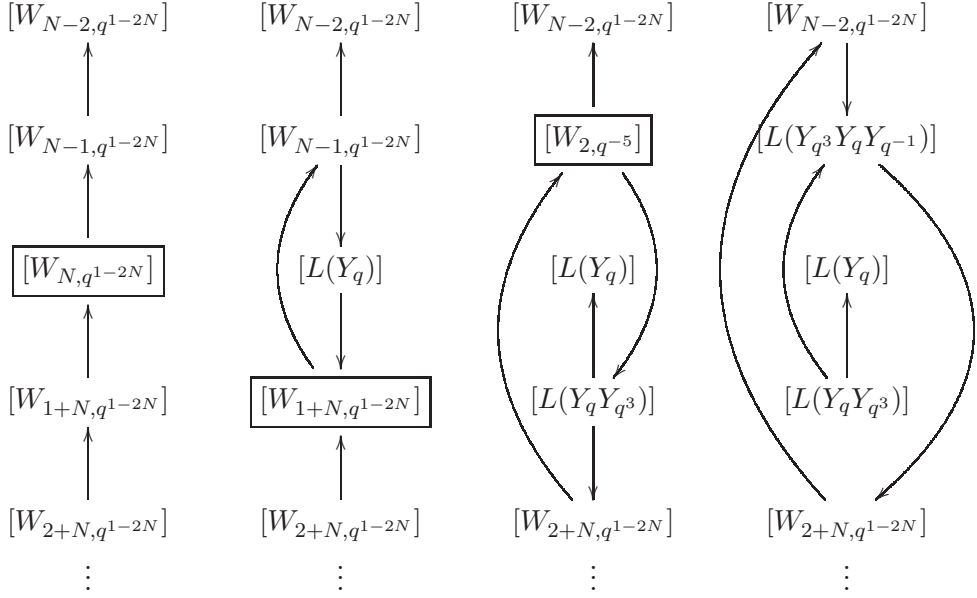
$$\xymatrix{ 
[W_{N-2,q^{1-2N}}]& [W_{N-2,q^{1-2N}}]& [W_{N-2,q^{1-2N}}]&\ar[d] [W_{N-2,q^{1-2N}}]
\\\ar[u] [W_{N-1,q^{1-2N}}]&\ar[u] [W_{N-1,q^{1-2N}}]\ar[d]&\ar[u] \fbox{$[W_{2,q^{-5}}]$}\ar@/^{2pc}/[dd]&[L(Y_{q^3}Y_{q}Y_{q^{-1}})]\ar@/^{4pc}/[ddd]
\\\ar[u]\fbox{$[W_{N,q^{1-2N}}]$}&\ar[d][L(Y_q)]&[L(Y_q)]&[L(Y_q)]
\\\ar[u][W_{1+N,q^{1-2N}}]&\ar@/^{2pc}/[uu]\fbox{$[W_{1+N,q^{1-2N}}]$}&[L(Y_qY_{q^3})]\ar[u]\ar[d]&[L(Y_qY_{q^3})]\ar[u]\ar@/^{2pc}/[uu]
\\\ar[u][W_{2+N,q^{1-2N}}]&\ar[u][W_{2+N,q^{1-2N}}]&[W_{2+N,q^{1-2N}}]\ar@/^{3pc}/[uuu]&[W_{2+N,q^{1-2N}}]\ar@/^{4pc}/[uuuu]
\\{}\save[]+<0cm,4ex>*{\vdots}\restore&{}\save[]+<0cm,4ex>*{\vdots}\restore&{}\save[]+<0cm,4ex>*{\vdots}\restore&{}\save[]+<0cm,4ex>*{\vdots}\restore}$$
The cluster variable corresponding to $[L(Y_q)]$ has its $q$-character which can be written in the form of Equation (\ref{laforme}) : 
\begin{eqnarray*}
\chi_q(L(Y_q)) &=& Y_q + Y_{q^3}^{-1} = \frac{\chi_q(W_{N-1,q^{1-2N}}) + \chi_q(W_{N+1,q^{1 - 2N}})}{\chi_q(W_{N,q^{1-2N}})}
\\[2mm]
&=&Y_{q^{-1}}^{-1}  \frac{1 + A_{q^{-2}}^{-1}(1 + A_{q^{-4}}^{-1}(1  
+ \cdots (1 + A_{q^{2(1-N)}}^{-1}))\cdots )}{1 + A_{1}^{-1}(1 + A_{q^{-2}}^{-1}(1  + \cdots (1 + A_{q^{2(1-N)}}^{-1}))\cdots )} 
\\[2mm]
&&+\ Y_q\frac{1 + A_{q^{2}}^{-1}(1 + A_{1}^{-1}(1  + \cdots (1 + A_{q^{2(1-N)}}^{-1}))\cdots )}{1 + A_{1}^{-1}(1 + A_{q^{-2}}^{-1}(1  + \cdots (1 + A_{q^{2(1-N)}}^{-1}))\cdots )}.
\end{eqnarray*}
We see as in the statement of Lemma \ref{prneg} that the monomial $Y_q$ maximal for $\preceq$ is dominant 
and so negative in the sense of Definition \ref{defineg}. 
In the limit $N\rightarrow +\infty$ (with the renormalized weights) we get the representations in $K_0(\mathcal{O}^-)$
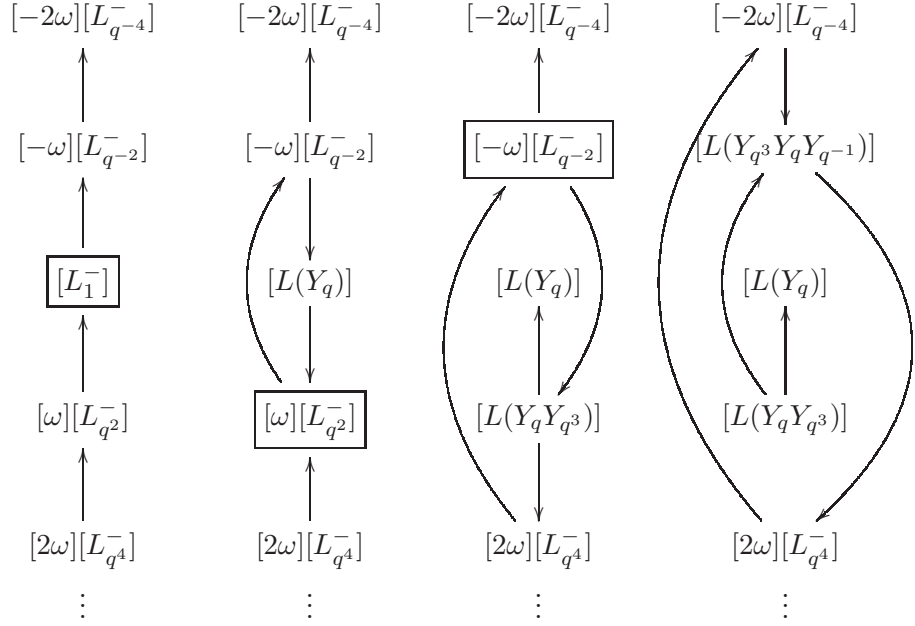
$$\xymatrix{ 
[-2\omega][L_{q^{-4}}^-]& [-2\omega][L_{q^{-4}}^-]& [-2\omega][L_{q^{-4}}^-]&\ar[d] [-2\omega][L_{q^{-4}}^-]
\\\ar[u] [-\omega][L_{q^{-2}}^-]&\ar[u] [-\omega][L_{q^{-2}}^-]\ar[d]&\ar[u] \fbox{$[-\omega][L_{q^{-2}}^-]$}\ar@/^{2pc}/[dd]&[L(Y_{q^3}Y_{q}Y_{q^{-1}})]\ar@/^{4pc}/[ddd]
\\\ar[u]\fbox{$[L_{1}^-]$}&\ar[d][L(Y_q)]&[L(Y_q)]&[L(Y_q)]
\\\ar[u][\omega][L_{q^{2}}^-]&\ar@/^{2pc}/[uu]\fbox{$[\omega][L_{q^{2}}^-]$}&[L(Y_qY_{q^3})]\ar[u]\ar[d]&[L(Y_qY_{q^3})]\ar[u]\ar@/^{2pc}/[uu]
\\\ar[u][2\omega][L_{q^4}^-]&\ar[u][2\omega][L_{q^4}^-]&[2\omega][L_{q^4}^-]\ar@/^{3pc}/[uuu]&[2\omega][L_{q^4}^-]\ar@/^{4pc}/[uuuu]
\\{}\save[]+<0cm,4ex>*{\vdots}\restore&{}\save[]+<0cm,4ex>*{\vdots}\restore&{}\save[]+<0cm,4ex>*{\vdots}\restore&{}\save[]+<0cm,4ex>*{\vdots}\restore
}$$
The duality again gives the representations in $K_0(\mathcal{O}^+)$ : 
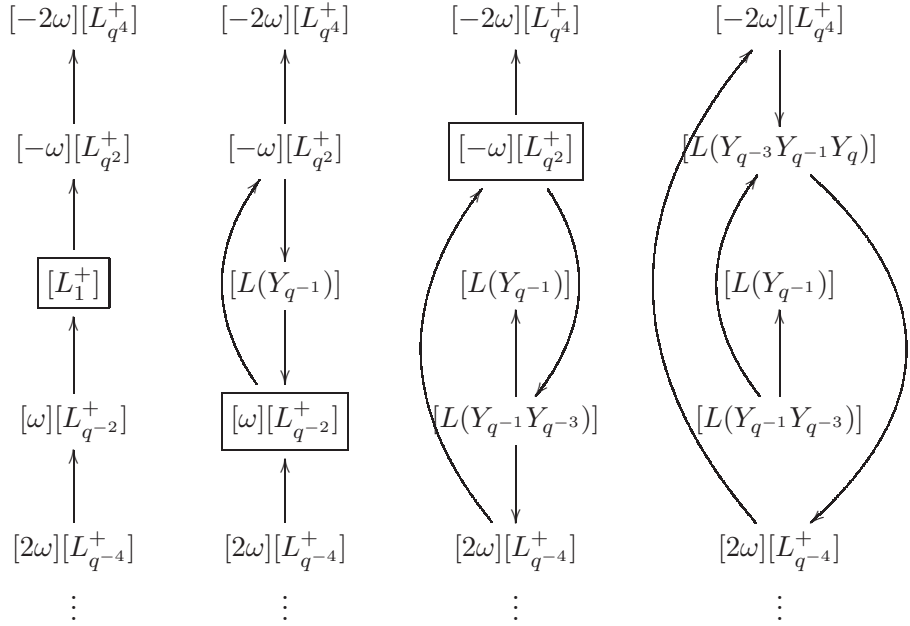
$$\xymatrix{ 
[-2\omega][L_{q^4}^+]& [-2\omega][L_{q^4}^+]& [-2\omega][L_{q^4}^+]&\ar[d] [-2\omega][L_{q^4}^+]
\\\ar[u] [-\omega][L_{q^2}^+]&\ar[u] [-\omega][L_{q^2}^+]\ar[d]&\ar[u] \fbox{$[-\omega][L_{q^2}^+]$}\ar@/^{2pc}/[dd]&[L(Y_{q^{-3}}Y_{q^{-1}}Y_q)]\ar@/^{4pc}/[ddd]
\\\ar[u]\fbox{$[L_{1}^+]$}&\ar[d][L(Y_{q^{-1}})]&[L(Y_{q^{-1}})]&[L(Y_{q^{-1}})]
\\\ar[u][\omega][L_{q^{-2}}^+]&\ar@/^{2pc}/[uu]\fbox{$[\omega][L_{q^{-2}}^+]$}&[L(Y_{q^{-1}}Y_{q^{-3}})]\ar[u]\ar[d]&[L(Y_{q^{-1}}Y_{q^{-3}})]\ar[u]\ar@/^{2pc}/[uu]
\\\ar[u][2\omega][L_{q^{-4}}^+]&\ar[u][2\omega][L_{q^{-4}}^+]&[2\omega][L_{q^{-4}}^+]\ar@/^{3pc}/[uuu]&[2\omega][L_{q^{-4}}^+]\ar@/^{4pc}/[uuuu]
\\{}\save[]+<0cm,4ex>*{\vdots}\restore&{}\save[]+<0cm,4ex>*{\vdots}\restore&{}\save[]+<0cm,4ex>*{\vdots}\restore&{}\save[]+<0cm,4ex>*{\vdots}\restore
}$$}
\end{example}

\begin{example}{\rm We now illustrate  the proof of Lemma \ref{prneg} by means of the mutation
of \ref{example-sl3}. Let us write the cluster
variables $\phi_N(\chi)$ :

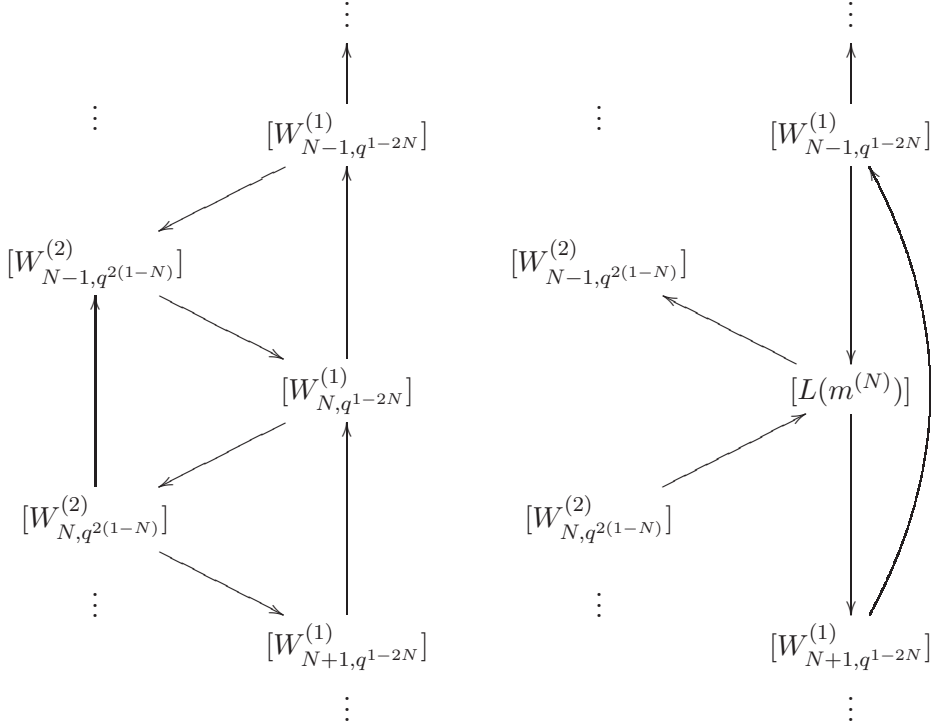
$$\xymatrix{ &{}\save[]+<0cm,2ex>*{\vdots}\restore&&{}\save[]+<0cm,2ex>*{\vdots}\restore
\\{}\save[]+<0cm,2ex>*{\vdots}\restore&\ar[u]\ar[dl] [W_{N-1,q^{1-2N}}^{(1)}] & {}\save[]+<0cm,2ex>*{\vdots}\restore  &\ar[u]\ar[dd]  [W_{N-1,q^{1-2N}}^{(1)}]
\\[W_{N-1,q^{2(1-N)}}^{(2)}]\ar[dr]&&[W_{N-1,q^{2(1-N)}}^{(2)}]&
\\&[W_{N,q^{1-2N}}^{(1)}]\ar[uu]\ar[dl]&&[L(m^{(N)})]\ar[ul]\ar[dd]
\\[ W_{N,q^{2(1-N)}}^{(2)}]\ar[uu]\ar[dr]&&[W_{N,q^{2(1-N)}}^{(2)}]\ar[ru]&
\\{}\save[]+<0cm,4ex>*{\vdots}\restore&[W_{N+1,q^{1-2N}}^{(1)}]\ar[uu]&{}\save[]+<0cm,4ex>*{\vdots}\restore&[W_{N+1,q^{1-2N}}^{(1)}]\ar@/_{2.5pc}/[uuuu]
\\&{}\save[]+<0cm,4ex>*{\vdots}\restore&&{}\save[]+<0cm,4ex>*{\vdots}\restore}$$
Here $m^{(N)} = Y_{1,q} (Y_{2,q^{2(1-N)}}Y_{2,q^{4-2N}}\cdots Y_{2,q^{-2}})$. 
The $q$-character corresponding to the cluster variable  $[L(m^{(N)})]$ can be written in the form of Equation (\ref{laforme}) : 
$$\chi_q(L(m^{(N)})) = \frac{\chi_q(W_{N-1,q^{1-2N}}^{(1)})\chi_q(W_{N,q^{2(1-N)}}^{(2)})}{\chi_q(W_{N,q^{1-2N}}^{(1)})} + \frac{\chi_q(W_{N+1,q^{1 - 2N}}^{(1)})\chi_q(W_{N-1,q^{2(1 - N)}}^{(2)})}{\chi_q(W_{N,q^{1-2N}}^{(1)})}$$
$$=Y_{1,q^{-1}}^{-1}(Y_{2,q^{2(1-N)}}Y_{2,q^{4-2N}}\cdots Y_{2,1}) (1 + \mathcal{A}_1^{(N)}) + Y_{1,q} (Y_{2,q^{2(1-N)}}Y_{2,q^{4-2N}}\cdots Y_{2,q^{-2}}) (1 + \mathcal{A}_2^{(N)}),$$
where $\mathcal{A}_1^{(N)}$ and $\mathcal{A}_2^{(N)}$ are formal power series in the 
$A_{1,q^a}^{-1}$, $A_{2,q^a}^{-1}$ ($a\in\CC^*$) without constant term. 
The monomial $\widetilde{m^{(N)}}$ is maximal for $\preceq$. 
Its limit for $N\rightarrow + \infty$ is $[-\omega_1]Y_{1,q}\Psib_{2,q^{-1}}^{-1}$, 
which is negative in the sense of Definition \ref{defineg}.}
\end{example}

\begin{example}{\rm In this example we check that the images of the initial cluster variables considered in the proof 
of Lemma \ref{prneg} do match.
 Let us consider type $B_2$ with the following initial seed and the
initial cluster variables replaced by the $W_{i,r}$ :
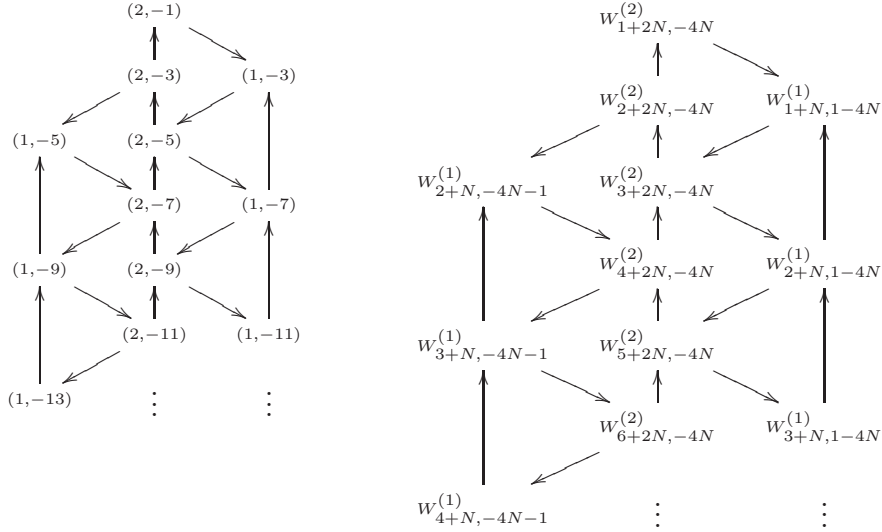
\[
\def\objectstyle{\scriptscriptstyle}
\xymatrix@-1.0pc{
&&&&\\
&&(\bdeux,-1)\ar[rd]
& 
\\
&&\ar[ld]\ar[u] (\bdeux,-3)& \ar[ld](\bun,-3)&
\\
&{(\bun,-5)}\ar[rd] &\ar[u] (\bdeux, -5) \ar[rd]&&
\\
&&\ar[u] \ar[ld](\bdeux,-7) &\ar[ld] (\bun,-7) \ar[uu]  &
\\
&(\bun,-9)\ar[uu]\ar[rd]    & \ar[u](\bdeux,-9)\ar[rd] &&
\\
&&(\bdeux, -11)\ar[u]\ar[ld] & (\bun,-11)\ar[uu] &
\\
&(\bun,-13)\ar[uu]& {}\save[]+<0cm,0ex>*{\vdots}\restore& {}\save[]+<0cm,0ex>*{\vdots}\restore&
}
\xymatrix@-1.0pc{
&&&&\\
&& W_{1+2N,-4N}^{(2)}\ar[rd]
& 
\\
&&\ar[ld]\ar[u] W_{2+2N,-4N}^{(2)}& \ar[ld]W_{1+N,1-4N}^{(1)}&
\\
&W_{2+N,-4N-1}^{(1)}\ar[rd] &\ar[u] W_{3+2N,-4N}^{(2)} \ar[rd]&&
\\
&&\ar[u] \ar[ld]W_{4+2N,-4N}^{(2)} &\ar[ld] W_{2+N,1-4N}^{(1)} \ar[uu]  &
\\
&W_{3+N,-4N-1}^{(1)}\ar[uu]\ar[rd]    & \ar[u]W_{5+2N,-4N}^{(2)}\ar[rd] &&
\\
&&W_{6+2N,-4N}^{(2)}\ar[u]\ar[ld] & W_{3+ N,1-4N}^{(1)}\ar[uu] &
\\
&W_{4+N,-4N-1}^{(1)}\ar[uu]& {}\save[]+<0cm,0ex>*{\vdots}\restore& {}\save[]+<0cm,0ex>*{\vdots}\restore&
}
\]
In the limit $N\rightarrow +\infty$ (with the renormalized weights) we get the images of the initial cluster variables :
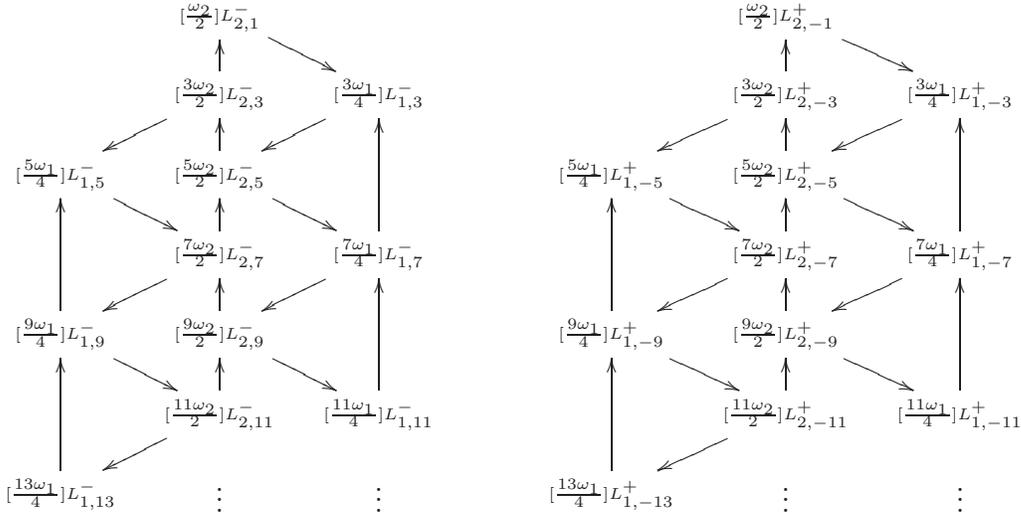
\[
\def\objectstyle{\scriptscriptstyle}
\xymatrix@-1.0pc{
&&&&\\
&& [\frac{\omega_2}{2}]L_{2,1}^-\ar[rd]
& 
\\
&&\ar[ld]\ar[u] [\frac{3\omega_2}{2}]L_{2,3}^-& \ar[ld][\frac{3\omega_1}{4}]L_{1,3}^-&
\\
&[\frac{5\omega_1}{4}]L_{1,5}^-\ar[rd] &\ar[u] [\frac{5\omega_2}{2}]L_{2,5}^- \ar[rd]&&
\\
&&\ar[u] \ar[ld][\frac{7\omega_2}{2}]L_{2,7}^- &\ar[ld] [\frac{7\omega_1}{4}]L_{1,7}^- \ar[uu]  &
\\
&[\frac{9\omega_1}{4}]L_{1,9}^-\ar[uu]\ar[rd]    & \ar[u][\frac{9\omega_2}{2}]L_{2,9}^-\ar[rd] &&
\\
&&[\frac{11\omega_2}{2}]L_{2,11}^-\ar[u]\ar[ld] & [\frac{11\omega_1}{4}]L_{1,11}^-\ar[uu] &
\\
&[\frac{13\omega_1}{4}]L_{1,13}^-\ar[uu]& {}\save[]+<0cm,0ex>*{\vdots}\restore& {}\save[]+<0cm,0ex>*{\vdots}\restore&
}
\xymatrix@-1.0pc{
&&&&\\
&& [\frac{\omega_2}{2}]L_{2,-1}^+\ar[rd]
& 
\\
&&\ar[ld]\ar[u] [\frac{3\omega_2}{2}]L_{2,-3}^+& \ar[ld][\frac{3\omega_1}{4}]L_{1,-3}^+&
\\
&[\frac{5\omega_1}{4}]L_{1,-5}^+\ar[rd] &\ar[u] [\frac{5\omega_2}{2}]L_{2,-5}^+ \ar[rd]&&
\\
&&\ar[u] \ar[ld][\frac{7\omega_2}{2}]L_{2,-7}^+ &\ar[ld] [\frac{7\omega_1}{4}]L_{1,-7}^+ \ar[uu]  &
\\
&[\frac{9\omega_1}{4}]L_{1,-9}^+\ar[uu]\ar[rd]    & \ar[u][\frac{9\omega_2}{2}]L_{2,-9}^+\ar[rd] &&
\\
&&[\frac{11\omega_2}{2}]L_{2,-11}^+\ar[u]\ar[ld] & [\frac{11\omega_1}{4}]L_{1,-11}^+\ar[uu] &
\\
&[\frac{13\omega_1}{4}]L_{1,-13}^+\ar[uu]& {}\save[]+<0cm,0ex>*{\vdots}\restore& {}\save[]+<0cm,0ex>*{\vdots}\restore&
}
\]}
\end{example}

\section{Conjectures and evidences}\label{secconj}

\subsection{A conjecture}

The concept of a monoidal categorification of a cluster algebra was introduced in 
\cite[Definition 2.1]{HL}. We say that a simple object $S$ of a monoidal category is 
{\em real} if $S\otimes S$ is simple. Let us recall that a cluster monomial
is a monomial in the cluster variables of a single cluster.

\begin{defi}\label{defmoncat}
Let $\mathcal{A}$ be a cluster algebra and let $\mathcal{M}$ be an abelian monoidal category.
We say that $\mathcal{A}$ is a monoidal categorification of $\mathcal{A}$
if there is an isomorphism between $\mathcal{A}$ and the Grothendieck ring of $\mathcal{M}$ 
such that the cluster monomials of $\mathcal{A}$ are the classes of all the 
real simple objects of $\mathcal{M}$  (up to invertibles).
\end{defi}

See \cite[Section 2]{HLpr} for a discussion on applications of monoidal categorifications.
In view of Theorem \ref{isom}, it is natural to formulate the following conjecture.

\begin{conj}\label{conjgene} The isomorphism of Theorem \ref{isom} defines a monoidal categorification, that is, the cluster
monomials in $\mathcal{A}(\Gamma)$ get identified with real simple objects in $\mathcal{O}_{2\ZZ}^+$ up to invertible representations.
\end{conj}

\begin{rem} 
{\rm
By using the duality in Proposition \ref{duality}, the statement of Theorem
\ref{isom} and of Conjecture \ref{conjgene} can also be formulated in terms of the category $\mathcal{O}^-_{2\ZZ}$.
}
\end{rem}

Note that Theorem \ref{stensor} implies that all cluster monomials of the initial seed are identified with 
real simple objects, more precisely with simple tensor products of positive prefundamental representations,
in agreement with Conjecture \ref{conjgene}.
To give other evidences supporting Conjecture \ref{conjgene}, we will use the results in the next subsection.

\subsection{Limiting characters}\label{limfin}

We will be using the dual category $\mathcal{O}^*$
considered in \cite{HJ} whose definition we now recall.

\begin{defi} Let $\mathcal{O}^*$ be the category of
  Cartan-diagonalizable $U_q(\bo)$-modules $V$ such that $V^*$ is in
  category $\mathcal{O}$.
\end{defi}

A $U_q(\mathfrak{b})$-module $V$ is said to be of lowest $\ell$-weight 
$\Psib\in P_\ell$ if there is $v\in V$ such that $V =U_q(\mathfrak{b})v$ 
and the following hold:
\begin{align*}
U_q(\bo)^- v = \CC v\,,
\qquad 
\phi_{i,m}^+v=\Psi_{i, m}v\quad (i\in I,\ m\ge 0)\,.
\end{align*}
For $\Psib\in P_\ell$, there exists up to isomorphism a unique simple $U_q(\bo)$-module
$L'(\Psib)$ of lowest $\ell$-weight $\Psi$. This module belongs to $\mathcal{O}^*$.
More precisely, we have:

\begin{prop}\label{dualweight}\cite{HJ} For $\Psib\in P_\ell$ we
  have $(L'(\Psib))^* \simeq L(\Psib^{-1})$.
\end{prop}

We can also define as in
\S\ref{qcarsec}
notions of characters and $q$-characters for $\mathcal{O}^*$.

 We now explain that characters (resp. $q$-characters) of certain simple objects
in the category $\mathcal{O}^-$ can be obtained as limits of characters (resp. $q$-characters) of finite-dimensional representations. This is known for negative prefundamental representations \cite{HJ}.

Let $L(\Psib)$ be a simple module whose highest $\ell$-weight can be written as a finite product
$$\Psib = [\omega] \times m \times \prod_{i\in I}\left(\prod_{r\geq -R_0}\Psib_{i,q^r}^{u_{i,q^r}}\right)$$
where $\omega\in P_\Q$, $R_0\geq 0$, $u_{i,q^r}\leq 0$ and $m\in\mathcal{M}$ is a dominant monomial.
For $R\geq R_0$, set 
\begin{equation}\label{formmr}M_R = m \prod_{i\in I}\left(\prod_{r\geq -R_0,\ r'\geq 0,\ r-2d_ir'\geq -R}Y_{i,q^{r-2d_ir'-d_i}}^{-u_{i,q^r}}\right).\end{equation}

\begin{thm}\label{normlimit} (1) We have the limit as formal power series :
$$\widetilde{\chi}_q(L(M_R))\underset{R\rightarrow +\infty}{\rightarrow} \widetilde{\chi}_q(L(\Psib)) 
\in \ZZ\left[\left[A_{i,a}^{-1}\right]\right]_{i\in I, a\in\CC^*}.$$
(2) We have $\widetilde{\chi}(L(\overline{\Psib}^{-1})) = \widetilde{\chi}(L(\Psib))$ and so we have the limit as formal power series :
$$\widetilde{\chi}(L(M_R))\underset{R\rightarrow +\infty}{\rightarrow} \widetilde{\chi}(L(\overline{\Psib}^{-1}))\in \ZZ[[-\alpha_i]]_{i\in I}.$$
\end{thm}

The proof of Theorem \ref{normlimit} is essentially the same as that of \cite[Theorem 6.1]{HJ}, so 
we just give an outline.

\begin{proof}  
First let us prove that the dimensions of
weight spaces of $\widetilde{L}(\Psib)$ are larger than those of $\widetilde{L}(M_R)$. 
Consider the tensor product
$$T = \widetilde{L}(\Psib)\otimes \widetilde{L}(M_R\Psib^{-1}).$$
By definition of $M_R$, the $\ell$-weight $M_R\Psib^{-1}$ is a product of $\Psib^+_{i,q^r}$ times $[\lambda]$ 
for some $\lambda\in P_\Q$, so by
Theorem \ref{stensor}, the module $\widetilde{L}(M_R\Psib^{-1})$ is a tensor product of positive fundamental representations. 
Moreover $T$ and $\widetilde{L}(M_R)$ have the same 
highest $\ell$-weight, so $\widetilde{L}(M_R)$ is a subquotient of~$T$.
By \cite[Theorem 4.1]{Fre2}, each $\ell$-weight of $\widetilde{L}(M_R)$ is the product of the highest
$\ell$-weight $M_R(\varpi(M_R))^{-1}$ by a product of
$A_{j,b}^{-1}$, $j\in i$, $b\in\CC^*$ .
Hence, by Theorem \ref{formuachar}, an $\ell$-weight of $T$ is an
$\ell$-weight of $\widetilde{L}(M_R)$ only if it is of the form
$$\Psib'\varpi(\Psib)^{-1}(\widetilde{M_R} \widetilde{\Psib}^{-1})$$
where $\Psib'$ is an $\ell$-weight of $L(\Psib)$ and $\widetilde{M_R}\widetilde{\Psib}^{-1}$ 
is the highest $\ell$-weight of $\widetilde{L}(M_R\Psib^{-1})$. We get the result for the dimensions.

Then we prove as in \cite[Section 4.2]{HJ} that we can define an inductive linear system
$$L(M_0)\rightarrow L(M_1)\rightarrow \cdots \rightarrow L(M_R)\rightarrow L(M_{R+1})\rightarrow \cdots$$
from the $L(M_R)$ so that we have the convergency of the action of the subalgebra $\widetilde{U}_q(\Glie)$ 
of $U_q(\Glie)$ generated by the $x_{i,r}^+$ and the $k_i^{-1}x_{i,r}^-$. 
We get a limiting representation of $\widetilde{U}_q(\Glie)$ from which one can construct
a representation $L_1$ of $U_q(\bo)$ in the category $\mathcal{O}$ and a representation $L_2$ of $U_q(\bo)$
in the category $\mathcal{O}^*$ \cite[Proposition 2.4]{HJ}. 
Moreover, $L_1$ (resp. $L_2$) is of highest (resp. lowest) $\ell$-weight $\Psib$ (resp. $\overline{\Psib}$). 

By construction, the normalized $q$-character of $L_1$ is the limit of the normalized $q$-characters 
$\widetilde{\chi}_q(L(M_R))$ as formal power series. Combining with the result of the first paragraph of this proof, 
the representation $L_1$ is necessarily simple isomorphic to $L(\Psib)$. We have proved the first statement in the Theorem.

Now, by construction $L_2^*$ is in the category $\mathcal{O}$ with highest $\ell$-weight $\overline{\Psib}^{-1}$ and 
satisfies $\widetilde{\chi}(L_2^*) = \widetilde{\chi}(L_1)$. To conclude, it suffices to prove that $L_2^*$ is irreducible. 
This is proved as in \cite[Theorem 6.3]{HJ}. 
\end{proof}

We have the following application:

\begin{thm}\label{simpledual} Let $L(\Psib)$ be a simple module in the category $\mathcal{O}^-$ such that $\widetilde{\Psib} = \Psib$. 
Then its image by $D^{-1}$ in $K_0(\mathcal{O}^+)$ is simple equal to $D^{-1}([L(\Psib)]) = [L(\overline{\Psib}^{-1})]$.
\end{thm}

\begin{proof}  From Example~\ref{exdual}~(i), the property is satisfied by negative 
prefundamental repre\-sentations.
Since these representations generate the fraction field of $K_0(\mathcal{O}^-)$, it suffices to show that the assignment
$[L(\Psib)]\mapsto [L(\overline{\Psib}^{-1})]$ for $\ell$-weights $\Psib$ satisfying $\Psib = \widetilde{\Psib}$ is multiplicative
(recall that $D$ is a morphism of $\mathcal{E}$-algebras).
Let us use the same notation as in the proof of Theorem \ref{normlimit} above.
For $L(\Psib)\simeq L_1$ and $L(\Psib')\simeq L_1'$ simple modules
in $\mathcal{O}^-$ with $\widetilde{\Psib} = \Psib$ and $\widetilde{\Psib'} = \Psib'$,
we have the corresponding modules $L_2$, $L_2'$ in $\mathcal{O}^*$. We consider the decomposition
$$[L(\Psib)\otimes L(\Psib')] = \sum_{\Psib'', \widetilde{\Psib''} = \Psib''}m_{\Psib''}[L(\Psib'')]$$
in $K_0(\mathcal{O}^-)$ with the $m_{\Psib''}\in\mathcal{E}$.
Each $\chi_q(L(\Psib''))$ is obtained as a limit as in Theorem \ref{normlimit}, and by construction the corresponding modules $L'(\Psib'')$
in the category $\mathcal{O}^*$ satisfy
$$[L_2\otimes L_2'] = \sum_{\Psib'', \widetilde{\Psib''} = \Psib''}  m_{\Psib''}' [L'(\overline{\Psib''})].$$
where each $m_{\Psib''}'\in\mathcal{E}$ is obtained from $m_{\Psib''}$ via the substitution $[\omega] \mapsto [-\omega]$.
This implies $[L_2^*\otimes (L_2')^*] = \sum_{\Psib'', \widetilde{\Psib''} = \Psib''} m_{\Psib''} [(L'(\overline{\Psib''}))^*]$,
that is, in view of Proposition~\ref{dualweight},
$$[L(\overline{\Psib}^{-1})]\otimes [L(\overline{\Psib'}^{-1})] =
\sum_{\Psib'', \widetilde{\Psib''} = \Psib''}m_{\Psib''}[L(\overline{\Psib''}^{-1})],$$
as required.
\end{proof}

\begin{example} {\rm
Applying $D$ to (\ref{firstmutation}), we get the following relation in $K_0(\mathcal{O}^-)$:
$$\left[L(\overline{\Psib}^{-1})\otimes L_{i,q^{-r}}^-\right] = \left[\bigotimes_{j,C_{j,i}\neq 0}L_{j,q^{-r+d_jC_{j,i}}}^-\right]
+ [-\alpha_i]\left[\bigotimes_{j,C_{j,i}\neq 0}L_{j,q^{-r-d_jC_{j,i}}}^-\right],$$
where 
$$\overline{\Psib}^{-1} = [-\omega_i]Y_{i,q^{-r+d_i}}\prod_{j,C_{j,i}<0}\Psib^{-1}_{j,q^{-r+d_jC_{j,i}}}.$$
}
\end{example}

\subsection{Proof of Conjecture \ref{conjgene} for $\Glie = \widehat{\sl}_2$} 
In this section we give an explicit description of all simple
modules in $\mathcal{O}^+$ and in $\mathcal{O}^-$ for $\Glie = \widehat{\sl}_2$.

A $q$-set is a subset of $\CC^*$ of the form $\{aq^{2r}\mid R_1\leq r\leq R_2\}$
for some $a\in\CC^*$ and $R_1\leq R_2\in\ZZ\cup \{-\infty,+\infty\}$.
The KR-modules $W_{k,a}$, $W_{k',b}$ are said to be in \emph{special position} if the union of $\{a,aq^2,\cdots, aq^{2(k-1)}\}$
and $\{b,bq^2, \cdots , bq^{2(k'-1)}\}$ is a $q$-set which contains both properly.
The KR-module $W_{k,a}$ and the prefundamental representation $L_b^+$ are said to be in special
position if the union of $\{a,aq^2,aq^4,\cdots , aq^{2(k-1)}\}$ and $\{bq, bq^3, bq^5,\cdots \}$ is a $q$-set which contains both properly.
Two positive prefundamental representations are never in special position.
Two representations are in \emph{general position} if they are not in special position.

The invertible elements in the category $\mathcal{O}^+$ are the $1$-dimensional representations $[\omega]$.

\begin{thm}\label{fact} Suppose that $\Glie = \widehat{\sl}_2$. The prime simple objects in the category $\mathcal{O}^+$ are the 
positive prefundamental representations and the KR-modules (up to invertibles). 
Any simple object in $\mathcal{O}^+$ can be factorized in a
unique way as a tensor product of prefundamental 
representations and KR-modules (up to permutation of the factors and to invertibles).
Moreover, such a tensor product is simple if and only all its factors are pairwise in general position.
\end{thm}

\begin{proof} As in the classical case of finite-dimensional representations, it is easy to check
that every positive $\ell$-weight has a unique factorization as a product of highest $\ell$-weights of
KR-modules and positive prefundamental representations in pairwise general position.
Hence it suffices to prove the equivalence in the last sentence. 
By \cite{CP}, the result is known for finite-dimensional representations.
Now, by using Section \ref{limfin}, this result implies that a tensor product with factors which are 
in general position is simple. 
Conversely, it is known that a tensor product of KR-modules which are in special position
is not simple. Also, it is easy to see that the tensor product of a KR-module and
a positive prefundamental representation which are in special position is not simple.
\end{proof}

\begin{rem}\label{remark7-8}
{\rm
(i) This is a generalization of the factorization of simple representations in $\mathcal{C}$ when 
$\Glie = \widehat{\sl}_2$ \cite{CP}.

(ii) This result for $\Glie = \widehat{\sl}_2$ implies that all simple objects in $\mathcal{O}^+$
are real and that their factorization into prime representations is unique.

(iii) In \cite{MY}, a factorization is proved for simple modules in $\hat{\mathcal{O}}$ when $\Glie = \widehat{\sl}_2$. 
But the factorization is not unique in the category $\hat{\mathcal{O}}$, see Remark~\ref{rem3.10}.

(iv) By Proposition \ref{duality}, our result implies a similar factorization in the category $\mathcal{O}^-$.

(v) The combinatorics of $q$-sets in pairwise general position is very similar to the combinatorics of triangulations of the $\infty$-gone 
studied in \cite{GG}, in relation with certain cluster structures of infinite rank. However, in \cite{GG} only arcs $(m,n)$ joining 
two integers $m$ and $n$ are considered, whereas we also allow arcs of the form $(m,+\infty)$ corresponding to positive prefundamental 
representations. Also, we are only interested in one mutation class, namely the mutation class of the initial triangulation
$\{(m,+\infty) \mid m\in\Z\}$.
}
\end{rem}

\begin{thm}\label{psl2} Conjecture \ref{conjgene} is true in the $\sl_2$-case.
\end{thm}

\begin{proof}
 Theorem \ref{fact} provides an explicit factorization of simple objects in $\mathcal{O}_{2\ZZ}^+$
into positive prefundamental representations and finite-dimensional KR-modules. In particular,
we get an explicit $q$-character formula for such a simple object and so a complete explicit description
of the Grothendieck ring $K_0(\mathcal{O}_{2\ZZ}^+)$. Besides the cluster algebra $\mathcal{A}(\Gamma)$
can be also explicitly described by using triangulations of the $\infty$-gone (see Remark~\ref{remark7-8}). 
Hence we can argue as in \cite[\S13.4]{HL0}.
\end{proof}

\subsection{Equivalence of conjectures}

In general, we have the following:

\begin{thm}\label{equi} Conjecture \ref{conjgene} is equivalent to \cite[Conjecture 5.2]{HL}.
\end{thm}

Combining with the recent results in \cite{Qin}, this would imply a part of Conjecture 
\ref{conjgene} for $ADE$ types, namely that all cluster monomials are classes of real simple objects.

As reminded in the introduction, \cite[Conjecture 5.2]{HL} states that $\mathcal{C}_\ZZ^-$ is the monoidal
categorification of a cluster algebra $\mathcal{A}(G^-)$. Note that $\mathcal{C}_\ZZ^-$ is a subcategory
of $\mathcal{O}_{2\ZZ}^-$ and $\mathcal{O}_{2\ZZ}^+$.

 \begin{proof} 
For $N > 0$, let $\mathcal{C}_N$ be the category of finite-dimensional $U_q(\Glie)$-modules $V$ satisfying 
$$[V]\in \ZZ[[V_{i,q^m}]]_{i\in I, -2dN - d_i\leq m < d(N+2) - d_i}\subset K_0(\mathcal{C}).$$ 
It is a monoidal category similar to the categories considered in \cite{HL0}. It contains the KR-module 
$W_{i,r,N}$ with highest monomial $M_{i,r,N}$ given in Equation (\ref{monkr}), where $i\in I$ and $-d(N+2) < r \leq 2dN$.
The Grothendieck ring $K_0(\mathcal{C}_N)$ has a cluster algebra structure with an initial seed consisting of these 
KR-modules $W_{i,r,N}$ (here we use the initial seed as in \cite{HL}).  We have established in the proof of 
Theorem \ref{isom} that $K_0(\mathcal{C}_N)\otimes \mathcal{E}$ may be seen as a subalgebra of 
$K_0(\mathcal{O}_{2\ZZ}^+)$ by using the identification of $z_{i,r}$ with the element defined in equation (\ref{newident}).
This induces embeddings $K_0(\mathcal{C}_N)\subset K_0(\mathcal{C}_{N+1})$ :
$$K_0(\mathcal{C}_1)\subset K_0(\mathcal{C}_2)\subset K_0(\mathcal{C}_3)\subset \cdots  \subset K_0(\mathcal{O}_{2\ZZ}^+)$$
which are not the naive embeddings obtained from the inclusion of categories $\mathcal{C}_N\subset \mathcal{C}_{N+1}$. The cluster monomials in $K_0(\mathcal{C}_N)$ corresponds now to cluster monomials in $K_0(\mathcal{O}_{2\ZZ}^+)$. 

Note that by Theorem \ref{simpledual}, we may consider indifferently the statement of Conjecture \ref{conjgene} for $\mathcal{O}^+_{2\ZZ}$ or for $\mathcal{O}^-_{2\ZZ}$.

Suppose that  \cite[Conjecture 5.2]{HL} is true. This implies that the cluster monomials in $K_0(\mathcal{C}_N)$ are
the real simple modules for any $N > 0$. Consider a cluster monomial in $\mathcal{O}_{2\ZZ}^+$. 
Then for $N$ large enough, we have
a corresponding real representation $V_N$ in $K_0(\mathcal{C}_N)$. 
The highest monomial of $V_N$ is a Laurent monomial in the $m_{i,r,N}$
of the form considered in Theorem \ref{normlimit}. 
By Theorem \ref{normlimit}, $\chi_q(V_N)$ converges to the $q$-character of a simple module 
$V$ in $\mathcal{O}^-_{2\ZZ}$ when $N\rightarrow +\infty$. 
Moreover $V$ is real as the $q$-character of $V\otimes V$ is obtained as a limit of simple $q$-character 
$\chi_q(V_N\otimes V_N)$ by Theorem \ref{normlimit}. 
Conversely, every real simple module $V$ in $\mathcal{O}^-_{2\ZZ}$ is obtained as such a limit simple modules $V_N$. 
Moreover since $V$ is real $V_N$ is real (for $\chi_q(V_N\otimes V_N)$ is an upper $q$-character of $V\otimes V$ 
in the sense of \cite[Corollary 5.8]{h3}). For $N$ large enough the modules $V_N$ correspond to the same cluster monomial 
which is therefore identified with $V$.

Conversely suppose that Conjecture \ref{conjgene} is true. Consider a cluster monomial $\chi$ in $K_0(\mathcal{C}_\ZZ^-)$. 
The cluster variables occuring in $\chi$ are produced via sequences of mutations from a finite number 
of KR-modules in the initial seed. By the proof of Proposition \ref{unsens}, there is a seed in $K_0(\mathcal{O}^-_{2\ZZ})$ 
containing these KR-modules (and the quiver of this seed has the same arrows joining the corresponding vertices). 
By our hypothesis $\chi$ is the class of a simple real module as an element of $K_0(\mathcal{O}^-_{2\ZZ}) \supset K_0(\mathcal{C}_\ZZ^-)$. Hence it is also real simple in $K_0(\mathcal{C}_\ZZ^-)$. Now consider
a real simple module $[V]$ in $K_0(\mathcal{C}_\ZZ^-)$. It corresponds to a cluster monomial in $K_0(\mathcal{O}^-_{2\ZZ})$
which is a cluster monomial in $K_0(\mathcal{C}_\ZZ^-)$ by the same arguments.
\end{proof}

\subsection{Web property theorem}

Let us prove the following generalization of the main result of \cite{h3}. 
If Conjecture~\ref{conjgene} holds, then the statement of the next theorem is a 
necessary condition for simple modules in the same seed.

\begin{thm} Let $S_1,\ldots, S_N$ be simple objects in $\mathcal{O}^+$ (resp. in $\mathcal{O}^-$). Then $S_1\otimes \cdots \otimes S_N$
is simple if and only if the tensor products $S_i\otimes S_j$ are simple for $i \le j$.
\end{thm}

\begin{proof} By Proposition \ref{duality}, it is equivalent to prove the statement in the category $\mathcal{O}^-$. 
Note that the ``only if'' part is clear. 
For the ``if'' part of the statement, we may assume without loss of generality that the zeros and poles of the 
highest $\ell$-weights of the $S_i$ are in $q^\ZZ$ (see (ii) in Remark \ref{ident}). 
For each simple module $S_i$, consider a corresponding simple finite-dimensional module $L(M_{R,i})$ as in Section \ref{limfin}. 
Since $S_i\otimes S_j$ is simple, there exists $R_1$ such that for $R\geq R_1$ the tensor product
$L(M_{R,i})\otimes L(M_{R,j})$ is simple. Indeed, by Theorem \ref{normlimit}, 
$\widetilde{\chi}_q(S_i\otimes S_j)$ is the limit of the $\widetilde{\chi}_q(L(M_{R,i}M_{R,j}))$. 
More precisely, there is $R_1$ such that for $R\geq R_1$, the image of $\widetilde{\chi}_q(S_i\otimes S_j)$ 
in $\ZZ[A_{i,q^r}^{-1}]_{i\in I, r \geq -R_1 + r_i}$ is equal to $\widetilde{\chi}_q(L(M_{R,i}M_{R,j}))$.
This implies that $\widetilde{\chi}_q(L(M_{R,i}M_{R,j})) = \widetilde{\chi}_q(L(M_{R,i}))\widetilde{\chi}_q(L(M_{R,j}))$.

Now, by \cite{h3},
$L(M_{R,1})\otimes \cdots \otimes L(M_{R,N})$ is simple isomorphic to $L(M_{R,1} \cdots M_{R,N})$.
This implies that the character of $S_1\otimes \cdots \otimes S_N$ is the same as the character
of the simple module with the same highest $\ell$-weight. Hence they are isomorphic.

\end{proof}

\begin{rem} 
{\rm This provides an alternative proof of Theorem \ref{fact}.
}
\end{rem}

\subsection{Another conjecture}

To conclude, let us state another general conjecture. Although the cluster algebra structure 
presented in this paper does not appear in the statement, this conjecture arises naturally
if we compare Theorem \ref{isom} with the results of \cite{HL}.

We consider a simple finite-dimensional representation $L(m)$ whose dominant
monomial $m$ satisfies $m\in \ZZ[Y_{i,aq^r}]_{(i,r)\in W, r\leq R}$ for a given
 $R\in\ZZ$. We have the corresponding truncated $q$-character \cite{HL0}
$$\chi_q^{\leq R}(L(m))\in m\ZZ[A_{i,q^r }^{-1}]_{i\in I, r \leq R - d_i}$$
which is the sum (with multiplicity) of the monomials $m'$ occurring in 
$\chi_q(L(m))$ satisfying $m(m')^{-1}\in\ZZ[A_{i,q^r }]_{i\in I, r \leq R - d_i}$.
As in the statement of Theorem \ref{relations}, we consider 
$$\chi_\ell^{\leq R}(L(m))\in \text{Frac}(K_0(\mathcal{O}^+))$$ 
obtained from $\chi_q^{\leq R}(L(m))$ by replacing each variable $Y_{i,a}$ by $[\omega_i]\ell_{i,aq_i^{-1}}\ell_{i,aq_i}^{-1}$.

Let us set 
$$W_R = \{(i,r)\in W | R\geq r > R - 2 d_i\}.$$
For $(i,r)\in W_R$, we set $u_{i,r}$ to be the maximum of $0$ and of the powers $u_{i,q^r}(m')$ of $Y_{i,q^r}$ in 
all monomials $m'$ occurring in $\chi_q^{\leq R}(L(m))$. Let 
$$\Psib_R = \prod_{(i,r)\in W_R}\Psib_{i,q^{r + d_i}}^{u_{i,r}}\text{ and }\Psib = m \Psib_R.$$ 
The representations $L(\Psib)$, $L(\Psib_R)$ are in the category $\mathcal{O}^+$. By Theorem \ref{stensor}, 
the representation $L(\Psib_R)$ is a simple tensor product of positive prefundamental representations.

\begin{conj}\label{genconj} We have the relation in $\text{Frac}(K_0(\mathcal{O}^+))$ :
$$\chi_\ell(L(\Psib )) = \chi_\ell^{\leq R}(L(m))  \prod_{(i,r)\in W_R}\ell_{i,q^{r+d_i}}^{u_{i,r}}
=\chi_\ell^{\leq R}(L(m)) \chi_\ell(L(\Psib_R)).$$
\end{conj}

\begin{rem} {\rm By taking the $q$-character, the statement is equivalent to the following $q$-character formula :
$$\chi_q(L(\Psib )) = \chi_q^{\leq R}(L(m)) \chi_q(L(\Psib_R)).$$}
\end{rem}

In some cases the conjecture is already proved :

\begin{example}\label{exconjvraie}{\rm
(i) In the case $\chi_q^{\leq R}(L(m)) = \chi_q(L(m))$, we have $u_{i,r} = 0$ for $(i,r)\in W_R$. 
Here the conjecture reduces to the generalized Baxter relations of Theorem \ref{relations}.

(ii) In the case $R = r + d_i$ and $m = Y_{i,q^{r-d_i}}$, we have $\chi_q(L(m))^{\geq R} = m (1 + A_{i,q^r}^{-1})$.
The conjecture reduces to the relations we have established in Formula (\ref{firstmutation}). 

(iii) As discussed above, it follows from Theorem \ref{equi} and from the main result of \cite{Qin},
that for $ADE$-types, all cluster monomials are classes of real simple objects. In particular
for $ADE$-types, Conjecture \ref{genconj} holds for all simple modules $L(m)$ which are
cluster monomials (for the cluster algebra structure defined in \cite{HL}). Indeed we may
assume that $R = 0$. It is proved in \cite{HL} that for any dominant monomial $m$,
the truncated $q$-character $\chi_q^{\leq 0}(L(m))$ is an element
of the cluster algebra $\A(G^-)$ defined in the proof of Proposition \ref{unsens}.
In $\chi_q^{\leq 0}(L(m))$ we perform the same substitution as in Theorem \ref{relations} above
(that is we apply the ring homomorphism $F$ of the proof of Proposition \ref{unsens}).
If we assume that $\chi_q^{\leq 0}(L(m))$ is a cluster variable of $\A(G^-)$, then it follows 
from the proof of Proposition \ref{unsens} that 
\[
y :=  F\left(\chi_q^{\leq 0}(L(m))\right) \prod_{(i,r)\in W_0} z_{i,q^{r + d_i}}^{u_{i,r}} 
\]
is a cluster variable in $\A(H^-)$. Using Theorem \ref{equi}, we deduce that $y$ is the $\ell$-character
of a simple module. Since $F$ is multiplicative, the argument readily extends to simple modules 
$L(m)$ such that $\chi_q^{\leq 0}(L(m))$ is a cluster monomial.
}
\end{example}

\bigskip
\small
\noindent
\begin{tabular}{ll}
David {\sc Hernandez}  
& Sorbonne Paris Cit\'e, Univ Paris Diderot,\\
&CNRS Institut de
  Math\'ematiques de Jussieu-Paris Rive Gauche UMR 7586,\\
& B\^atiment Sophie Germain, Case 7012,
75205 Paris Cedex 13, France\\
&Institut Universitaire de France,\\
& email : {\tt david.hernandez@imj-prg.fr}\\
[5mm]
Bernard {\sc Leclerc}  & Universit\'e de Caen Basse-Normandie,\\
&CNRS UMR 6139 LMNO, 14032 Caen, France\\
&email : {\tt bernard.leclerc@unicaen.fr}
\end{tabular}


\begin{thebibliography}{ASMASM}

\bibitem[Ba]{Baxter} 
R.J.~Baxter,
Partition Function of the Eight-Vertex Lattice Model, 
\emph{Ann. Phys.} {\bf 70} (1971).


\bibitem[Be]{bec} {J. Beck}, 
\newblock Braid group action and quantum affine algebras, 
\newblock {\em Comm. Math. Phys.} {\bf 165} 
(1994), 555--568.

\bibitem[BHK]{BHK} {V. Bazhanov, A. Hibberd and S. Khoroshkin}, 
\newblock Integrable structure of $W_3$ conformal field theory,
quantum Boussinesq theory and boundary affine Toda theory,
\newblock {\em Nucl. Phys. B} {\bf 622} (2002), 475--547.

\bibitem[BJMST]{BJMST} {H. Boos, M. Jimbo , T. Miwa, F. Smirnov and Y. Takeyama}
\newblock Hidden Grassmann Structure in the XXZ Model II: Creation Operators
\newblock {\em Comm. Math. Phys.} {\bf 286} (2009), 875--932

\bibitem[BLZ]{BLZ} {V.V. Bazhanov, S.L. Lukyanov and
  A.B. Zamolodchikov},
\newblock Integrable structure of conformal field
theory. II. Q-operator and DDV equation,
\newblock {\em Comm. Math. Phys.} {\bf 190} (1997), 247--278.

\bibitem[C]{Cha2} {V. Chari}, 
\newblock Minimal affinizations of representations of quantum groups:
the rank 2 case,
\newblock {\em Publ. Res. Inst. Math. Sci.} {\bf 31} (1995), no. 5,
873--911.

\bibitem[CP]{CP} {V. Chari and A. Pressley},
\newblock {\it A Guide to Quantum Groups},
\newblock Cambridge University Press, Cambridge, 1994.

\bibitem[Da]{da} {I. Damiani}, 
\newblock La $\mathcal{R}$-matrice pour les alg\`ebres quantiques de
type affine non tordu,
\newblock {\em Ann. Sci. Ecole Norm. Sup.} {\bf 31}
(1998), 493--523.

\bibitem[Dr]{Dri2} {V. Drinfel'd}, 
\newblock A new realization of Yangians and of 
quantum affine algebras, 
\newblock {\em Soviet Math. Dokl.} {\bf 36}
(1988), 212--216.  

\bibitem[FH]{FH} {E. Frenkel and D. Hernandez},
\newblock Baxter's Relations and Spectra of Quantum Integrable Models,
\newblock {\em Duke Math. J.} {\bf 164} (2015), no. 12, 2407--2460.

\bibitem[FM]{Fre2} {E. Frenkel and E. Mukhin}, 
\newblock Combinatorics of $q$-characters of finite-dimensional 
representations of quantum affine algebras, 
\newblock {\em Comm. Math. Phys.} {\bf 216}
(2001), 23--57.

\bibitem[FR]{Fre} {E. Frenkel and N. Reshetikhin}, 
\newblock The $q$-characters of representations of quantum 
affine algebras and deformations of $W$-Algebras, 
\newblock in Recent Developments in Quantum Affine Algebras and 
related topics, 
\newblock 
{\em Contemp. Math.} {\bf 248} (1999), 163--205 (corrected version
available at arXiv:math/9810055).

\bibitem[FZ1]{FZ1} S. Fomin and A. Zelevinsky, 
Cluster algebras I: Foundations,
{\em J. Amer. Math. Soc.} {\bf 15} (2002), 497--529.

\bibitem[FZ2]{FZsurvey} S. Fomin and A. Zelevinsky, 
Cluster algebras: notes for the CDM-03 conference, 
in Current developments in mathematics, 2003,  1--34, Int. Press,
Somerville, MA, 2003.

\bibitem[GL]{GL} Grabowski J. and Launois S.,
\newblock Graded quantum cluster algebras and an application to quantum Grassmannians,
\newblock {\em Proc. Lond. Math. Soc.} {\bf 109} (2014), no. 3, 697--732.

\bibitem[GG]{GG} J. Grabowski and S. Gratz, 
\newblock Cluster algebras of infinite rank,
\newblock {\em J. Lond. Math. Soc.} {\bf 89} (2014), no. 2, 337-363.

\bibitem[GSV]{GSV} M. Gekhtman, M. Shapiro and A. Vainshtein, 
\emph{Cluster algebras and Poisson geometry}, 
AMS Math. Survey and Monographs {\bf 167}, AMS 2010.

\bibitem[H1]{Her04}{D. Hernandez}, 
\newblock Representations of quantum affinizations and fusion product, 
\newblock {\em Transform. Groups} {\bf 10} (2005), no. 2, 163--200. 

\bibitem[H2]{Herbis}{D. Hernandez},
\newblock On minimal affinizations of representations of quantum groups,
\newblock {\em  Comm. Math. Phys. {\bf 276} (2007), no. 1, 221--259}

\bibitem[H3]{h3} {D. Hernandez}, 
\newblock Simple tensor products, 
\newblock {\em Invent. Math.} {\bf 181} 
(2010), 649--675.

\bibitem[H4]{h4} {D. Hernandez},
\newblock Asymptotical representations, spectra of quantum integrable systems and cluster algebras, 
in {\em Oberwolfach Rep.} {\bf 12} (2015), 1385--1447.

\bibitem[HJ]{HJ} {D. Hernandez and M. Jimbo},
\newblock Asymptotic representations and Drinfeld rational fractions,
\newblock {\em  Compos. Math.} {\bf 148} (2012), no. 5, 1593--1623.

\bibitem[HL1]{HL0}{D. Hernandez and B. Leclerc},
\newblock Cluster algebras and quantum affine algebras,
\newblock {\em Duke Math. J.} {\bf 154} (2010), no. 2, 265--341.

\bibitem[HL2]{HLpr}{D. Hernandez and B. Leclerc},
\newblock Monoidal categorifications of cluster algebras of type A and D
\newblock in Symmetries, Integrable systems and Representations,  Springer Proc. Math. Stat. {\bf 40} (2013), 175--193.

\bibitem[HL3]{HL} {D. Hernandez and B. Leclerc},
\newblock A cluster algebra approach to $q$-characters of
Kirillov-Reshetikhin modules,
\newblock {\em J. Eur. Math. Soc.} {\bf 18} (2016), 1113--1159. 

\bibitem[Ka]{ka}{V. Kac},
\newblock {\it Infinite Dimensional Lie Algebras}, 3rd ed.,
\newblock Cambridge University Press, Cambridge, 1990.

\bibitem[Ko]{Ko}
{T. Kojima}, 
\newblock The Baxter's $Q$ operator for the $W$ algebra $W_N$, 
\newblock {\em J. Phys.} \textbf{A 41}, (2008) 355206.

\bibitem[KKKO1]{kkko1}{S-J. Kang, M. Kashiwara, M. Kim and S-J. Oh},
\newblock Monoidal categorification of cluster algebras,
\newblock{Preprint arXiv:1412.8106}

\bibitem[KKKO2]{kkko2}{S-J. Kang, M. Kashiwara, M. Kim and S-J. Oh},
\newblock Monoidal categorification of cluster algebras II,
\newblock{Preprint arXiv:1502.06714}

\bibitem[KQ]{KQ}{Y. Kimura and F. Qin}, 
\newblock Graded quiver varieties, quantum cluster algebras and dual canonical basis, 
\newblock {\em Adv. Math.}, vol {\bf 262} (2014), 261--312

\bibitem[MY]{MY} {E. Mukhin and C. Young},
\newblock Affinization of category O for quantum groups,
\newblock {\em Trans. Amer. Math. Soc.}, vol {\bf 366}, pp. 4815--4847 

\bibitem[N1]{Nab}{H. Nakajima},
\newblock Quiver varieties and t-analogs of q-characters of quantum affine algebras, 
\newblock {\em Annals of Math.} {\bf 160} (2004), 1057--1097.

\bibitem[N2]{N}{H. Nakajima},
\newblock Quiver varieties and cluster algebras, 
\newblock {\em Kyoto J. Math.} {\bf 51} (2011), no. 1, 71--126

\bibitem[Q]{Qin} {F. Qin},
\newblock Triangular bases in quantum cluster algebras and monoidal categorification conjectures,
\newblock Preprint arXiv:1501.04085.


\end{thebibliography}
\end{document}